\documentclass[11pt]{article}
\usepackage{amsfonts}
\usepackage{amssymb}
\usepackage{amsthm}
\usepackage{amsmath}
\usepackage{graphicx}
\usepackage{empheq}
\usepackage{indentfirst}
\usepackage{cite} 
\usepackage{mathrsfs}
\usepackage{cases}
\usepackage{graphics}
\usepackage{xcolor}  
\usepackage{amsmath,bm}

\newtheorem{theorem}{\textbf{Theorem}}[section]
\newtheorem{lemma}{\textbf{Lemma}}[section]
\newtheorem{proposition}{\textbf{Proposition}}[section]
\newtheorem{corollary}{\textbf{Corollary}}[section]
\newtheorem{remark}{\textbf{Remark}}[section]
\newtheorem{definition}{\textbf{Definition}}[section]
\allowdisplaybreaks[4] 
\def\be{\begin{equation}}
\def\ee{\end{equation}}
\def\bea{\begin{eqnarray}}
\def\eea{\end{eqnarray}}
\def\bt{\begin{theorem}}
\def\et{\end{theorem}}
\def\bl{\begin{lemma}}
\def\el{\end{lemma}}
\def\br{\begin{remark}}
\def\er{\end{remark}}
\def\bp{\begin{proposition}}
\def\ep{\end{proposition}}
\def\bc{\begin{corollary}}
\def\ec{\end{corollary}}
\def\bd{\begin{definition}}
\def\ed{\end{definition}}

\begin{document}

\title{Attractors for the Navier--Stokes--Cahn--Hilliard System with Chemotaxis and Singular Potential  in 2D}

\author{
	Jingning He
	\footnote{School of Mathematics, Hangzhou Normal University, Hangzhou 311121, P.R. China. Email: hejingning@hznu.edu.cn
	}}
\date{\today}
\maketitle


\begin{abstract}
\noindent We analyze the long-time behavior of solutions to a Navier--Stokes--Cahn--Hilliard system with chemotaxis effects and a solution-dependent mass source term. The fluid velocity $\bm{v}$ satisfies the Navier--Stokes system, the phase field variable $\varphi$ satisfies a convective Cahn--Hilliard equation with a singular potential (e.g., the Flory–Huggins type), the nutrient density $\sigma$ satisfies an advection-diffusion-reaction. For the initial boundary value problem in 2D, we prove the  existence of the global attractor in a suitable phase space. Furthermore, we obtain the existence of an exponential
attractor, and it can be implied that the global attractor is of finite fractal dimension. 
\end{abstract}

\section{Introduction}
\setcounter{equation}{0}
\noindent
We investigate the following Navier--Stokes--Cahn--Hilliard type system
\begin{subequations}
\begin{alignat}{3}
&\partial_t  \bm{ v}+\bm{ v} \cdot \nabla  \bm {v}-\textrm{div} (2  \eta(\varphi) D\bm{v} )+\nabla p=(\mu+\chi \sigma)\nabla \varphi,\label{f3.c} \\
&\textrm{div}\ \bm{v}=0,\label{f3.c1}\\
&\partial_t \varphi+\bm{v} \cdot \nabla \varphi=\Delta \mu-M(x,\varphi),\label{f1.a} \\
&\mu=A\varPsi'(\varphi)-B\Delta \varphi-\chi \sigma,\label{f4.d} \\
&\partial_t \sigma+\bm{v} \cdot \nabla \sigma= \Delta (\sigma+\chi(1-\varphi)) +S, \label{f2.b}
\end{alignat}
\end{subequations}
in $\Omega\times(0,T)$. Here, $\Omega \subset\mathbb{R}^2$ is a bounded domain with smooth boundary $\partial\Omega$ and $T>0$ is a given final time of arbitrary magnitude. The system is subject to the following boundary conditions
\begin{alignat}{3}
&\bm{v}=\mathbf{0},\quad {\partial}_{\bm{n}}\varphi={\partial}_{\bm{n}}\mu={\partial}_{\bm{n}}\sigma=0,\qquad\qquad &\textrm{on}& \   \partial\Omega\times(0,T),
\label{boundary}
\end{alignat}
and the initial conditions
\begin{alignat}{3}
&\bm{v}|_{t=0}=\bm{v}_{0}(x),\ \ \varphi|_{t=0}=\varphi_{0}(x), \ \ \sigma|_{t=0}=\sigma_{0}(x), \qquad &\textrm{in}&\ \Omega.
\label{ini0}
\end{alignat}
Here, $\bm{n}=\bm{n}(x)$ denotes the unit outward normal vector on $\partial\Omega$.

System \eqref{f3.c}--\eqref{f2.b} can be regarded as a simplified version of the general thermodynamically consistent diffuse interface model in \cite{LW} where the authors study a two-component incompressible fluid mixture including a chemical species subject to diffusion and some significant transport mechanisms, like advection and chemotaxis (see also \cite{Sitka} and the references cited therein). In this system, the
state variables of the system are volume-averaged velocity $\bm{v}$, the (modified) pressure $p$, the difference in volume fractions of the binary mixture $\varphi$, the concentration of the chemical species $\sigma$, the chemical potential $\mu$ associated to $(\varphi, \sigma)$. Here, $D\bm{v}=\frac{1}{2}(\nabla\bm{ v}+(\nabla\bm{ v}) ^ \mathrm{T})$ denotes the symmetrized velocity gradient. The potential $\varPsi$ denotes the
physically relevant free energy density  given by the logarithmic type \cite{CH}:
\be
\varPsi (r)=\frac{\theta}{2}\big[(1-r)\ln(1-r)+(1+r)\ln(1+r)\big]+\frac{\theta_{0}}{2}(1-r^2),\quad \forall\, r\in[-1,1],
\label{pot}
\ee
with $0<\theta<\theta_{0}$ (see, e.g., \cite{CH,CMZ}), which is also known as Flory--Huggins potential.  

The source term $M(x,\varphi)$ in \eqref{f1.a} (see (H4)
below) is related to biological mechanisms like proliferation, elimination of cells (see e.g., \cite{L2022}). Further discussions on biologically relevant mass source terms can be found in \cite{Mi19} and the
references cited therein. The source term $M(x,\varphi)$ can be viewed as a general form of  Oono's type  $\alpha(\varphi-c_0)$ with $\alpha\geq 0$, $c_0\in(-1,1)$, which represents nonlocal interactions (e.g. reversible chemical reaction) between the two components (cf. \cite{GGM2017,Mi11}). Assuming that the components $\eta_1$, $\eta_2>0$ are viscosities of the two homogeneous fluids, the mean viscosity is modeled by the concentration dependent term $\eta=\eta(\varphi)$, for instance, a typical form is the linear combination:
\be
\eta(\varphi)=\eta_1\frac{1+\varphi}{2}+\eta_2\frac{1-\varphi}{2}.
\label{vis}
\ee
 The positive constant $A$ is the surface tension coefficient, and the positive coefficient $B$ is related to thickness of the interfacial layer (i.e., the diffuse interface). The constant $\chi$ is related to certain specific transport mechanisms such as chemotaxis and active transport in the context of tumor growth modelling (see e.g., \cite{GLSS, GL17e}). As it has been shown in \cite{H,LW}, the following basic energy law holds for the problem \eqref{f3.c}--\eqref{ini0}
\begin{align}
& \frac{d}{dt} \int_{\Omega} \Big[ \frac{1}{2}|\bm{v}|^2+A\varPsi(\varphi)+\frac{B}{2}|\nabla \varphi|^2+\frac{1}{2}|\sigma|^2+\chi\sigma(1-\varphi) \Big] dx \notag \\
& \qquad +\int_{\Omega} \Big[ 2\eta(\varphi)|D\bm{v}|^2 +m(\varphi) |\nabla \mu|^2+|\nabla(\sigma+\chi(1-\varphi))|^2\Big] dx\nonumber\\
&\quad  =\int_\Omega \left[-M(x,\varphi)\mu+(-\mathcal{C} h(\varphi) \sigma +S)(\sigma +\chi(1-\varphi))\right] dx,
\label{BEL}
\end{align}
which plays a significant role in the study of its global well-posedness.

The system \eqref{f3.c}--\eqref{ini0} with a regular potential was first investigated in \cite{LW}. The authors of \cite{LW} obtained the existence of global weak solutions in dimension two and three, the existence and uniqueness of global strong solutions in dimension two under an artificial assumption on the coefficients $A$ and $\chi$. This extra assumption was removed in the recent work \cite{H}, where the author considered a singular potential like \eqref{pot} and proved the existence of global weak solutions in two and three dimensions. Then the authors in \cite{H1} proves the existence and regularity properties of the unique global strong solution for problem \eqref{f3.c}--\eqref{ini0} in two dimensions.   

Neglecting the interaction due to nutrient and assuming $\alpha=0$, our problem \eqref{f3.c}--\eqref{f4.d} reduces to the well-known ``Model H" for two incompressible and
viscous Newtonian fluids \cite{HH, Gur}. This model has been extensively investigated in the literature under various hypotheses on mobility, fluid viscosity and potential function. For instance, for the Navier--Stokes--Cahn--Hilliard system, we refer to \cite{B,GG2010,ZWH} for the regular potential case, and to \cite{A2009,B,GMT} when the potential is singular (e.g., logarithmic). We refer to \cite{BGM,MT} for  the system with nonlocal effect due to Oono's interaction.  Concerning the long-time behavior of the NSCH system
 with regular potential, the existence of global and
exponential attractors was proved in \cite{GG2010}. The lower bound for the global attractor can be found in \cite{GG2011}, and the existence of a pullback exponential attractor was studied \cite{BG2014}. For the logarithmic potential case, we refer to \cite{A2009} for the stability of local minima of the energy and the existence of the global attractor, and to the recent work \cite{GT} for the existence of the exponential attractor. Neglecting the fluid interaction in system \eqref{f3.c}--\eqref{f4.d}, the well-posedness of a Cahn--Hilliard type system with chemotaxis was proved in \cite{GL17e}, the case with Dirichlet boundary conditions was investigated in \cite{GL17}. Neglecting chemotaxis and active transport effects, the existence of global attractors of some related diffuse interface models under a choice of the  regular potential was proved in \cite{FGR,MRS}. Concerning the existence of exponential attractors for the Cahn--Hilliard equation, we refer to \cite{MW,MZ04} and the references therein. 

Our aim is to continue the study of \cite{H,H1} and establish the long-time behavior of problem \eqref{f3.c}--\eqref{ini0} with the choice of a singular potential like \eqref{pot} in two dimensions from
the viewpoint of infinite dimensional dynamical systems. Based on the well-posedness results in \cite{H,H1} and the energy law \eqref{BEL} of the system, we define a semigroup map on a suitable phase-space. We prove the existence of a compact absorbing set in the phase-space and the strong continuity of the semigroup, which implies the existence of the global attractor (cf. \cite{T2001}). Then, we obtain the existence of exponential attractors by showing the so-called smoothing property (cf. \cite{MZ08}). It is worth mentioning
that the separation property (that is, $\varphi$ stays away from the pure phases $\pm1$ from a certain time on) makes the Cahn--Hilliard equation with singular potential can be regarded as a parabolic equation with a globally Lipschitz potential. Finally, we prove
some regularity properties of the global attractor and its finite fractal dimension.

 The remaining part of this paper is organized as follows. In Section \ref{pm}, we state the functional settings and introduce the main results. In Section \ref{ws}, we introduce the
 dynamical system related to problem \eqref{f3.c}--\eqref{f2.b} and prove the existence of the global attractor. Section \ref{phs} is devoted to the existence of an exponential
 attractor.

\section{Main Results}\label{pm}
\setcounter{equation}{0}
\subsection{Preliminaries}
Let $\Omega \subset\mathbb{R}^2$ be a bounded domain with smooth boundary $\partial\Omega$. Suppose that $T>0$ is a fixed final time. For the standard Lebesgue and Sobolev spaces, we use the notations $L^{p} := L^{p}(\Omega)$ and $W^{k,p} := W^{k,p}(\Omega)$ for any $p \in [1,+\infty]$ and $k > 0$, equipped with the norms $\|\cdot\|_{L^{p}}$ and $\|\cdot\|_{W^{k,p}}$. We use $H^{k} := W^{k,2}$ and the norm $\|\cdot\|_{H^{k}}$ when $p = 2$. We denote the norm and inner product on $L^{2}(\Omega)$ by $\|\cdot\|$ and $(\cdot,\cdot)$, respectively.
We denote the dual space of a Banach space $X$ by $X'$, and denote the duality pairing between $X$ and $X'$  by
$\langle \cdot,\cdot\rangle_{X',X}$. For any given interval $J$ of $\mathbb{R}^+$, we introduce
the function space $L^p(J;X)$ with $p\in [1,+\infty]$, which
consists of Bochner measurable $p$-integrable
functions by taking values in the Banach space $X$.
We denote the vectorial space $X^d$ by the boldface letter $\bm{X}$, which consists of the product structure.

For every $f\in H^1(\Omega)'$,  $\overline{f}$ denotes its generalized mean value over $\Omega$ satisfying
$\overline{f}=|\Omega|^{-1}\langle f,1\rangle_{(H^1)',\,H^1}$. If $f\in L^1(\Omega)$, then its mean is simply given by $\overline{f}=|\Omega|^{-1}\int_\Omega f \,dx$.
  We recall the Poincar\'{e}--Wirtinger inequality
\begin{equation}
\label{poincare}
\|f-\overline{f}\|\leq C_P\|\nabla f\|,\quad \forall\,
f\in H^1(\Omega),
\end{equation}
where the constant $C_P$ depends only on $\Omega$. Then we
notice that $f\to (\|\nabla f\|^2+|\overline{f}|^2)^\frac12$ is an equivalent norm on $H^1(\Omega)$. We introduce the space $L^2_{0}(\Omega):=\{f\in L^2(\Omega):\overline{f} =0\}$ for $L^2$ functions with zero mean and denote
$H^2_{N}(\Omega):=\{f\in H^2(\Omega):\,\partial_{\bm{n}}f=0 \ \textrm{on}\  \partial \Omega\}$ for $H^2$ functions subject to homogeneous Neumann boundary condition. We recall the Ladyzhenskaya and Agmon's inequalities in two dimensions, which will be used in the subsequent estimates
\begin{align*}
&\|f\|_{L^4}\leq C\|f\|_{H^1}^\frac12\|f\|^\frac12,\qquad \forall\,f\in H^1(\Omega),\\
&\|f\|_{L^\infty}\leq  C\|f\|_{H^2}^\frac12\|f\|^\frac12,\qquad \forall\,f\in H^2(\Omega).
\end{align*}

Then we recall some function spaces related to the Navier--Stokes equations (see e.g., \cite{S}). Suppose that $\bm{C}^\infty_{0,\mathrm{div}}(\Omega)$ is the space of divergence free vector fields in $(C^\infty_0(\Omega))^2$. Define $\bm{L}^2_{0,\mathrm{div}}(\Omega)$ and $\bm{H}^1_{0,\mathrm{div}}(\Omega)$ as the closure of $\bm{C}^\infty_{0,\mathrm{div}}(\Omega)$ with respect to the $\bm{L}^2$ and $\bm{H}^1_0$ norms, respectively. The space $\bm{H}^1_{0,\mathrm{div}}(\Omega)$ is equipped with the scalar product
 $(\bm{u},\bm{v})_{\bm{H}^1_{0,\mathrm{div}}}:=(\nabla \bm{u},\nabla \bm{v})$ for all $\bm{u},\, \bm{v}  \in {\bm{H}^1_{0,\mathrm{div}}(\Omega)}$ and the norm $\|\bm{u}\|_{\bm{H}^1_{0,\mathrm{div}}}=\|\nabla \bm{u}\|$. The classical Korn's inequality
 entails that $\|\nabla \bm{u}\| \leq \sqrt{2} \|D\bm{u}\|\leq \sqrt{2}\|\nabla \bm{u}\|$ for all $\bm{u}\in \bm{H}^1_{0,\mathrm{div}}(\Omega)$.
 It is well known that $\bm{L}^2(\Omega)$ can be decomposed into $\bm{L}^2_{0,\mathrm{div}}(\Omega)\oplus\bm{G}(\Omega)$, where $\bm{G}(\Omega):=\{\bm{f}\in\bm{L}^2(\Omega): \exists\, g\in H^1(\Omega),\ \bm{f}=\nabla g\}$. For any function $\bm{f} \in \bm{L}^2(\Omega)$, the Helmholtz--Weyl decomposition holds:
$\bm{f}=\bm{f}_{0}+\nabla g$ where  $\bm{f}_{0} \in \bm{L}^2_{0,\mathrm{div}}(\Omega)$ and $\nabla g \in \bm{G}(\Omega)$. Consequently, we can define the Helmholtz--Leray projection onto the space of divergence-free functions $\bm{P}:\bm{L}^2(\Omega)\to \bm{L}^2_{0,\mathrm{div}}(\Omega)$ such that $\bm{P}(\bm{f})=\bm{f}_{0}$. We recall the Stokes operator $\bm{S}: \bm{H}^1_{0,\mathrm{div}}(\Omega)\cap\bm{H}^2(\Omega)\to\bm{L}^2_{0,\mathrm{div}}(\Omega)$ satisfying
\be
(\bm{S}\bm{u},\bm{\zeta})=(\nabla \bm{u},\nabla\bm{\zeta}),\quad  \forall\, \bm{\zeta} \in \bm{H}^1_{0,\mathrm{div}}(\Omega),\nonumber
 \ee
 with domain $D(\bm{S})= \bm{H}^1_{0,\mathrm{div}}(\Omega)\cap\bm{H}^2(\Omega)$ (see e.g., \cite[Chapter III]{S}). Besides, there exists a constant $C>0$ such that $\|\bm{u}\|_{\bm{H}^2}\leq C \|\bm{S} \bm{u}\|$ for any $\bm{u}\in \bm{H}^1_{0,\mathrm{div}}(\Omega)\cap\bm{H}^2(\Omega)$. The operator $\bm{S}$ is a canonical isomorphism from $\bm{H}^1_{0,\mathrm{div}}(\Omega)$ to $\bm{H}^1_{0,\mathrm{div}}(\Omega)'$.
 We denote its inverse map by $\bm{S}^{-1}:\bm{H}^1_{0,\mathrm{div}}(\Omega)'\to\bm{H}^1_{0,\mathrm{div}}(\Omega)$. For any $\bm{f}\in \bm{H}^1_{0,\mathrm{div}}(\Omega)'$, there is a unique $\bm{u}=\bm{S}^{-1}\bm{f}\in\bm{H}^1_{0,\mathrm{div}}(\Omega)$ satisfying
\be
(\nabla\bm{S}^{-1}\bm{f},\nabla \bm{\zeta})=\langle\bm{f},\bm{\zeta}\rangle_{(\bm{H}^1_{0,\mathrm{div}})',\,\bm{H}^1_{0,\mathrm{div}}},\quad \forall\, \bm{\zeta} \in \bm{H}^1_{0,\mathrm{div}}(\Omega),\nonumber
\ee
and $\|\nabla\bm{S}^{-1}\bm{f}\|=\langle\bm{f},\bm{S}^{-1}\bm{f} \rangle_{(\bm{H}^1_{0,\mathrm{div}})',\,\bm{H}^1_{0,\mathrm{div}}}^{\frac{1}{2}}$ is an equivalent norm on $\bm{H}^1_{0,\mathrm{div}}(\Omega)'$.
Finally, we recall the following regularity result (see e.g.,  \cite[Chapter III, Theorem 2.2.1]{S} and \cite[Appendix B]{GMT}):
\bl \label{stokes}
 \rm Let $\Omega$ be a bounded smooth domain. For any $\bm{f} \in \bm{L}^2_{0,\mathrm{div}}(\Omega)$,
there exists a unique pair $\bm{u}\in \bm{H}^1_{0,\mathrm{div}}(\Omega)\cap\bm{H}^2(\Omega)$ and $p\in H^1(\Omega)\cap L_0^2(\Omega)$ satisfying $-\Delta \bm{u}+\nabla p=\bm{f}$ a.e. in $\Omega$, that is, $\bm{u}=\bm{S}^{-1}\bm{f}$. Moreover,
\begin{align*}
&\|\bm{u}\|_{\bm{H}^2}+\|\nabla p\|\le C\|\bm{f}\|,\quad
\|p\|\le C \|\bm{f}\|^\frac12\|\nabla \bm{S}^{-1}\bm{f}\|^\frac12,
 \end{align*}
where the positive constant $C$  may depend on $\Omega$ but is independent of $\bm{f}$.
 \el

\subsection{Assumptions and well-posedness}
\noindent
We introduce some basic hypotheses for problem  \eqref{f3.c}--\eqref{ini0} (cf. \cite{H,GMT,GGW}).
\begin{enumerate}
	\item[(H1)]\label{seta} The fluid viscosity $\eta$ satisfies  $\eta \in C^{2}(\mathbb{R})$  and
	$\eta_{*} \leq \eta(r)\leq \eta^*$ for $r \in \mathbb{R}$, where $\eta_{*}<\eta^*$ are positive constants.
	\item[(H2)]\label{item:as} The singular potential $\varPsi$ belongs to the class of functions $C([-1,1])\cap C^{3}(-1,1)$ and has the form
	\begin{equation}
	\varPsi(r)=\varPsi_{0}(r)-\frac{\theta_{0}}{2}r^2,\nonumber
	\end{equation}
	satisfying
	\begin{equation}
	\lim_{r\to \pm 1} \varPsi_{0}'(r)=\pm \infty ,\quad \text{and}\ \  \varPsi_{0}''(r)\ge \theta>0,\quad \forall\, r\in (-1,1),\nonumber
	\end{equation}
	where $\theta_0-\theta>0$. We define $\varPsi_{0}(r)=+\infty$ for any $r\notin[-1,1]$.
	Moreover, there exists $\epsilon_0\in(0,1)$ such that $\varPsi_{0}''$ is nondecreasing in $[1-\epsilon_0,1)$ and nonincreasing in $(-1,-1+\epsilon_0]$.
	\item[(H3)] The convex function $\varPsi_0$ satisfies 
	\be \varPsi_0^{\prime \prime}(z) \leq C \mathrm{e}^{C\left|\varPsi_0^{\prime}(z)\right|}, \quad \forall\, z \in(-1,1). \label{de2}\ee
	 for some positive constant  $C$.
	\item[(H4)] $M: \Omega \times[-1,1] \rightarrow \mathbb{R}$ is given as $M(x, s)=a(x)(m(s)+b(x))$, for some $m \in C^{0,1}([-1,1])$, $a, b \in L^{\infty}(\Omega)$ satisfying
	$-a(x) \geq 0$ for all $x \in \Omega$ and there exists a non-empty measurable subset $U \subset \Omega$, as well as a positive constant $a_0$ such that
	$$
	a(x) \geq a_0 \text { for all } x \in U.
	$$
	There exists a positive constant $\varepsilon_0<1$ such that $\|b\|_{L^{\infty}} \leq 1-\varepsilon_0$.
	
	It holds that
	$$
	m(1)\left\{\begin{array} { l l } 
	{ \geq 1 } & { \text { if } b \neq 0 , } \\
	{ > 0 } & { \text { if } b \equiv 0 , }
	\end{array} \quad m ( - 1 ) \left\{\begin{array}{ll}
	\leq-1 & \text { if } b \neq 0, \\
	<0 & \text { if } b \equiv 0 .
	\end{array}\right.\right.
	$$
	\item[(H5)]\label{sco} The coefficients $A,\ B,\  \chi,\ \alpha,\ c_0$ are  prescribed  constants satisfying
	\be
	A>0,\ \ B>0,\ \  \chi \in \mathbb{R},\ \ \alpha\geq 0,\ \ c_0\in(-1,1). \nonumber
	\ee
	\item[(H6)] The functions $S$ satisfies 
	\be
	 \|S\|_{L^1(0,+\infty, L^1(\Omega))}+ \sup_{t\in [0,+\infty)}\int_t^{t+1}\|S(s)\|^2\,ds\leq S_0,
	\label{ass1}
	\ee
	where $S_0>0$ is an arbitrary but fixed constant.
\end{enumerate}
\begin{remark}
 The logarithmic potential \eqref{pot} satisfies the assumptions $\mathrm{(H2)}$ and $\mathrm{(H3)}$.
\end{remark}
 Then we report the well-posedness results.  When the spatial dimension is two, the existence and uniqueness of global weak solutions have been proven in \cite[Theorems 2.1, 2.2]{H}.:
\begin{theorem}[Weak well-posedness] \label{maind}
Let $\Omega \subset \mathbb{R}^{2}$ and $T \in (0,+\infty)$. Assume that the hypotheses (H1), (H2) and (H4)--(H5) are satisfied. Assume that the initial data satisfy $\bm {v}_{0} \in \bm {L}^2_{0,\mathrm{div}}(\Omega)$, $\varphi_{0}\in H^1(\Omega)$, $\sigma_{0}\in L^2(\Omega)$ with $\|  \varphi_{0} \|_{L^{\infty}} \le 1$ and
$|\overline{\varphi_{0}}|<1$. The system \eqref{f3.c}--\eqref{ini0} exists a weak solution $(\bm{v},\varphi,\mu,\sigma)$ on $[0,T]$, which fulfills the following properties
\begin{align}
&\bm{v} \in L^{\infty}(0,T;\bm{L}^2_{0,\mathrm{div}}(\Omega)) \cap L^{2}(0,T;\bm{H}^1_{0,\mathrm{div}}(\Omega))\cap  H^{1}(0,T;\bm{H}^1_{0,\mathrm{div}}(\Omega)')\notag,\\
&\varphi \in L^{\infty}(0,T;H^1(\Omega))\cap L^{4}(0,T;H^2_{N}(\Omega))\cap L^2(0,T;W^{2,q}(\Omega)) \cap H^{1}(0,T;H^1(\Omega)'),\notag \\
&\mu \in   L^{2}(0,T;H^1(\Omega)),\notag \\
&\sigma  \in L^{\infty}(0,T;L^2(\Omega))\cap L^{2}(0,T;H^1(\Omega)) \cap H^{1}(0,T;H^1(\Omega)'),\notag\\
&\varphi\in L^{\infty}(\Omega\times (0,T))\ \textrm{and}\ \ |\varphi(x,t)|<1\ \ \textrm{a.e.\ in}\ \Omega\times(0,T),\notag
\end{align}
where $q\in [2, +\infty)$, and
\begin{subequations}
	\begin{alignat}{3}
	&\left \langle\partial_t  \bm{ v},\bm{\zeta}\right \rangle_{(\bm{H}^1_{0,\mathrm{div}})',\,\bm{H}^1_{0,\mathrm{div}}}+(\bm{ v} \cdot \nabla  \bm {v},\bm{ \zeta})+(  2\eta(\varphi) D\bm{v},D\bm{ \zeta}) =((\mu+\chi \sigma)\nabla \varphi,\bm {\zeta}),\notag \\
	&\left \langle \partial_t \varphi,\xi\right \rangle_{(H^1)',\,H^1}+(\bm{v} \cdot \nabla \varphi,\xi)=- (\nabla \mu,\nabla \xi)-(M(x,\varphi),\xi),\notag \\
	&\left \langle\partial_t \sigma,\xi\right \rangle_{(H^1)',\,H^1}+(\bm{v} \cdot \nabla \sigma,\xi) + (\nabla \sigma,\nabla \xi) = \chi ( \nabla \varphi,\nabla \xi) + (S,\xi), \notag
	\end{alignat}
\end{subequations}
a.e. in $(0,T)$ for all $\bm {\zeta} \in \bm{H}^1_{0,\mathrm{div}}$ and $\xi \in H^1(\Omega)$, where $ \mu=A\varPsi'(\varphi)-B\Delta \varphi-\chi \sigma$ a.e. in $\Omega\times (0,T)$. Moreover, the initial conditions in \eqref{ini0} are satisfied.

 Consider two groups of initial data satisfying  $(\bm{v}_{0i},\varphi_{0i},\sigma_{0i})\in\bm{L}^2_{0,\mathrm{div}}(\Omega)\times H^1(\Omega)\times L^2(\Omega)$ with $\left \|  \varphi_{0i}\right \|_{L^{\infty}} \le 1$ and  $|\overline{\varphi}_{0i}|<1$, $i=1,\, 2$.
The global weak solutions $(\bm{v}_{1},\varphi_{1},\sigma_{1})$,   $(\bm{v}_{2},\varphi_{2 },\sigma_{2})$ to problem \eqref{f3.c}--\eqref{ini0} on $[0,T]$ with initial data $(\bm{v}_{0i},\varphi_{0i},\sigma_{0i})$, $i=1,\, 2$,
satisfy the following continuous dependence estimate:
\be
W(t)\le C\left(\frac{W(0)}{C}\right )^{\exp\big(-C\int_{0}^{t}Z(s)\, ds\big)},\quad \forall\, t\in [0,T],\label{uni}
\ee
where
\begin{align*}
W(t)&= \frac{1}{2}\|\nabla\bm{S}^{-1}[\bm{v}_1(t)-\bm{v}_2(t)]\|^2+\frac{1}{2}\|\varphi_1(t)-\varphi_2(t) \|_{(H^1)'}^2  +\frac{1}{2}\|\sigma_1(t)-\sigma_2(t) \|_{(H^1)'}^2\\
&\quad +|\overline{\varphi}_1(t)-\overline{\varphi}_2(t)|,\\
Z(t)&= \|\nabla \bm{v}_{1}(t)\|^2 +\|\nabla \bm{v}_{2}(t)\|^2 +\|\varphi_{1}(t)\|_{W^{2,3}}^2 +\|\varphi_{2}(t)\|_{W^{2,3}}^2 +\|\varphi_{1}(t)\|_{H^{2}}^4\notag\\
&\quad +\|\varPsi'(\varphi_{1})\|_{L^1}+\|\varPsi'(\varphi_{2})\|_{L^1}+\|\sigma_2(t)\|_{H^1}^2+1,
\end{align*}
and  {\color{black}{the constant $C$ may depend on the norm of the initial data}}, $\Omega$, $T$ and coefficients of the system. In particular, the global weak solution $(\bm{v},\varphi,\mu,\sigma)$ to problem \eqref{f3.c}--\eqref{ini0} is unique.
\end{theorem}
 The following global strong well-posedness result can be found in \cite{H1}.
\begin{theorem}[Strong well-posedness] \label{2main}
	Let $\Omega \subset \mathbb{R}^{2}$ be a bounded smooth domain and $T>0$. Suppose that the hypotheses (H1)--(H4) are satisfied. For any $ \bm{v}_{0} \in \bm{H}^1_{\mathrm{0,div}}(\Omega)$, $\varphi_{0} \in H^{2}_N(\Omega)$, $\sigma_{0}\in H^1(\Omega)$ with $\left\|\varphi_{0}\right\|_{L^{\infty}} \leq 1,\ \left|\overline{\varphi_{0}}\right|<1$ and
	$\mu_{0}=A\varPsi^{\prime}\left(\varphi_{0}\right) -B\Delta \varphi_{0} -\chi\sigma_0\in H^1(\Omega)$,  there exists a unique global strong
	solution $(\bm{v}, p, \varphi, \mu, \sigma)$ to problem \eqref{f3.c}--\eqref{ini0} on $[0, T]$ such that
	\begin{align*}
	&\bm{v} \in C\left([0, T] ; \bm{H}^1_{\mathrm{0,div}}(\Omega)\right) \cap L^{2}\left(0, T ; \bm{H}^1_{\mathrm{0,div}}(\Omega)\cap\bm{H}^2(\Omega)\right) \cap H^{1}\left(0, T ;\bm{L}^2_{0,\mathrm{div}}(\Omega)\right), \\
	&p\in L^2(0,T;H^1(\Omega)),\\
	&\varphi \in C_w\left([0, T] ; H^3(\Omega)\right) \cap L^2(0,T;H^4(\Omega))\cap H^{1}(0, T ; H^1(\Omega)),\\
	& \mu\in C([0, T] ; H^1(\Omega)) \cap L^{2}\left(0, T ; H^{3}(\Omega)\right) \cap H^{1}\left(0, T ; H^1(\Omega)^{\prime}\right),\\
	&\sigma   \in C([0,T] ;H^1(\Omega))\cap L^{2}(0,T;H^2(\Omega)) \cap H^1(0,T;L^2(\Omega)),\\
	&\varphi\in L^{\infty}(\Omega\times (0,T))\ \textrm{and}\ \ |\varphi(x,t)|<1\ \ \textrm{a.e.\ in}\ \Omega\times(0,T)\notag
	\end{align*}
	The strong solution satisfies the equations \eqref{f3.c}--\eqref{f2.b} a.e. in $\Omega \times(0, T)$, the boundary conditions \eqref{boundary} a.e. on $\partial\Omega\times(0,T)$ as well as the initial conditions \eqref{ini0} a.e. in $\Omega$. Consider two sets of initial data satisfying  $\bm{v}_{0i}\in\bm{H}^1_{0,\mathrm{div}}(\Omega)$, $\varphi_{0i}\in H^2_N(\Omega)$, $\sigma_{0i}\in H^1(\Omega)$ with $\left \|  \varphi_{0i}\right \|_{L^{\infty}} \le 1$, $|\overline{\varphi_{0i}}|<1$, and 	$\mu_{0i}=A\varPsi^{\prime}\left(\varphi_{0i}\right) -B\Delta \varphi_{0i} -\chi\sigma_{0i}\in H^1(\Omega)$.
	Denote the corresponding global strong solutions to problem \eqref{f3.c}--\eqref{ini0} by $(\bm{v}_{i},\varphi_{i},\mu_i,\sigma_{i})$, $i=1,\, 2$. Then we have the following estimate:
	\begin{align}
	&\|\boldsymbol{v}_1(t)-\boldsymbol{v}_2(t)\|^{2}
	+\|\varphi_1(t)-\varphi_2(t)\|^{2}_{H^1}
	+\|\sigma_1(t)-\sigma_2(t)\|^{2}\nonumber \\
	&\quad \le C_T\Big(\|\boldsymbol{v}_{01}-\boldsymbol{v}_{02}\|^{2}
	+\|\varphi_{01}-\varphi_{02}\|^{2}_{H^1}
	+\|\sigma_{01}-\sigma_{02}\|^{2} \Big)\label{conti1}
	\end{align}
	for all $t\in [0,T]$, where $C_T$ is a positive constant depending on norms of the initial data, coefficients of the system, $\Omega$ and $T$.
\end{theorem}
\subsection{Statement of main results}
For any given constants
$$m_1\in [0,1),\quad  c_0\in [-m_1,m_1],\quad \text{and}\ m_2\geq 0,$$
we define the phase space
\be
\mathcal{X}_{m_1,m_2}=\big\{(\bm{v},\varphi,\sigma)\in \bm{L}^2(\Omega)\times H^1(\Omega)\times L^2(\Omega)\,:\, \|\varphi\|_{L^\infty}\leq 1,\ |\overline{\varphi}|\leq m_1,\ |\overline{\sigma}|\leq m_2\big\} \label{phase}
\ee
 endowed with the metric
$
\mathrm{d}((\bm{v}_1,\varphi_1,\sigma_1),(\bm{v}_2,\varphi_2,\sigma_2))=
\|\bm{v}_1-\bm{v}_2\|+\|\varphi_1-\varphi_2\|_{H^1}+\|\sigma_1-\sigma_2\|
$.
We can check easily that $\mathcal{X}_{m_1,m_2}$ is a complete metric space.

We first prove the existence of a global attractor in $\mathcal{X}_{m_1,m_2}$:
\begin{theorem}[Global attractor]\label{attr}
	Assume that the hypotheses (H1)--(H6) are satisfied. Problem \eqref{f1.a}--\eqref{ini0} generates a strongly continuous
	semigroup $\mathcal{S}(t): \mathcal{X}_{m_1,m_2}\to \mathcal{X}_{m_1,m_2}$ such that for all $t\geq 0$ it holds
	$$\mathcal{S}(t)(\bm{v}_0,\varphi_0,\sigma_0) =(\bm{v}(t),\varphi(t),\sigma(t)) ,$$ where $(\bm{v},\varphi, \sigma)$ is the unique global weak solution proved in Theorem \ref{maind} corresponding to the initial datum $(\bm{v}_0,\varphi_0, \sigma_0)\in \mathcal{X}_{m_1,m_2}$. Moreover, the dynamical system $(\mathcal{S}(t), \mathcal{X}_{m_1,m_2})$ has the connected global attractor $\mathcal{A}_{m_1,m_2}\subset \mathcal{X}_{m_1,m_2}$ such that\\
	1. $\mathcal{A}_{m_1,m_2}$ is a compact set in $\mathcal{X}_{m_1,m_2}$;\\
	2. $\mathcal{A}_{m_1,m_2}$ is invariant in $\mathcal{X}_{m_1,m_2}$, i.e. $S(t) \mathcal{A}_{m_1,m_2}=\mathcal{A}_{m_1,m_2}$, for every $t \geq 0$;\\
	3. $\mathcal{A}_{m_1,m_2}$ is an attracting set for $S(t)$ on $\mathcal{X}_{m_1,m_2}$, i.e. for every bounded set $\mathcal{B} \subset \mathcal{X}_{m_1,m_2}$,
		$$
	\lim_{t\to 0}\mathrm{dist}_{\bm{L}^2\times H^1\times L^2}(\mathcal{S}(t)\mathcal{B},\mathcal{A}_{m_1,m_2})=0,\quad \forall\,t\geq 0,
	$$
	where $\mathrm{dist}_{\bm{L}^2\times H^1\times L^2}(\cdot,\cdot,\cdot)$ denotes the Hausdorff semidistance between two sets with respect to the metric of $\bm{L}^2(\Omega)\times H^1(\Omega)\times L^2(\Omega)$.
	
	 Furthermore, we mention that $\mathcal{A}_{m_1,m_2}$ is bounded in $\bm{H}^1(\Omega)\times H^3(\Omega)\times H^1(\Omega)$.
\end{theorem}
{\color{black}{
		Next, we prove that the dynamical system $(\mathcal{S}(t), \mathcal{X}_{m_1,m_2})$ possesses an exponential attractor $\mathcal{M}_{m_1,m_2}$.
		\begin{theorem}[Exponential attractor]\label{eattr}
			Assume that the hypotheses (H1)--(H6) are satisfied. Moreover, we assume that  $\varPsi$ belongs to the class of functions $ C^{4}(-1,1)$. The dynamical system  $(\mathcal{S}(t), \mathcal{X}_{m_1,m_2})$ has an exponential attractor $\mathcal{M}_{m_1,m_2}\subset \mathcal{X}_{m_1,m_2}$, which is bounded in $\bm{H}^1(\Omega)\times H^3(\Omega)\times H^1(\Omega)$ and satisfies the following properties:
			\begin{itemize}
				\item[$\mathrm{(1)}$] $\mathcal{M}_{m_1,m_2}$ is positively invariant in $\mathcal{X}_{m_1,m_2}$ satisfying $\mathcal{S}(t)\mathcal{M}_{m_1,m_2}\subset \mathcal{M}_{m_1,m_2}$ for all $t\geq 0$;
				\item[$\mathrm{(2)}$] $\mathcal{M}_{m_1,m_2}$ has finite fractal dimension in $\mathcal{X}_{m_1,m_2}$;
				\item[$\mathrm{(3)}$] $\mathcal{M}_{m_1,m_2}$ is an exponentially attracting set, that is, there holds a constant $\omega > 0$ such that, for every
				bounded set $\mathcal{B} \subset \mathcal{X}_{m_1,m_2}$, there is a positive constant $J$ satisfying
				$$
				\mathrm{dist}_{\bm{L}^2\times H^1\times L^2}(\mathcal{S}(t)\mathcal{B},\mathcal{M}_{m_1,m_2})\leq Je^{-\omega t},\quad \forall\,t\geq 0.
				$$
			\end{itemize}
\end{theorem}}}

 \section{Global attractor}\label{ws}
 \setcounter{equation}{0}
 We prove the existence of a global attractor for system \eqref{f1.a}--\eqref{ini0}, that is, in a suitable phase space, the global attractor is a unique compact set which is invariant under the semigroup generated by the evolution system and is an absorbing set for any bounded set when time goes to infinity.

Before starting the proof, we obtain the following uniform estimates which not only hold in the two-dimensional case but also in the three-dimensional case.
 \begin{lemma}\label{lm1} Assume that the hypotheses (H1), (H2) and (H4)--(H6) are satisfied.
The weak solution to problem \eqref{f3.c}--\eqref{ini0} exists on  $[0,+\infty)$ and satisfies the following uniform-in-time estimates:
 	\begin{align}
 		&\|\bm{v}(t)\|^2 +\|\varphi(t)\|_{H^1}^2+\|\sigma(t)\|^2 \notag \\
 		&\quad +\int_{t}^{t+1}\big(\|\nabla\bm{v}(s)\|^2
 		+\|\varphi(s)\|_{H^2}^4 + \|\nabla \mu(s)\|^2+ \|\nabla \sigma(s)\|^2\big)\,  d s \le C,\quad \forall\, t\geq 0,
 		\label{uni1}\\
 		&\int_{t}^{t+1}\big( \left\|\partial_{t} \varphi(s)\right\|_{(H^1)'}^{2} +\left\| \partial_{t} \sigma(s)\right\|_{(H^1)'}^2\big)\,ds
 		\leq C,\quad \forall\, t\geq 0,\label{uni2}
 	\end{align}
 	where the constant $C$ depends on $\|\bm{v}_0\|_{\bm{H}^1}$, $\|\mu_0\|_{H^1}$, $\|\sigma_0\|_{H^1}$, $\|\varphi_0\|_{H^1}$, $\int_\Omega \varPsi(\varphi_0)dx$, coefficients of the system and $\Omega$, but not on $t$.
 
 \end{lemma}
 
 \textbf{Proof}. The existence of a global weak solution $(\bm{v}, \varphi, \mu, \sigma)$ on an arbitrary time interval $[0,T]$ can be obtained by applying a similar argument in the proof of \cite[Theorems 2.1]{H}, where the author used a semi-Galerkin method such that the approximate solution satisfies 
 \be
 \|\varphi(t)\|_{L^\infty}\leq 1,\quad t\geq 0.\label{inif}
 \ee 
  Thus, we only need to derive \textit{a priori} estimates that are uniform in time, and 
 we keep in mind that \eqref{inif} holds in this process.
 
 Integrating \eqref{f1.a} over $\Omega$, we get
 \be
 \frac{d}{dt}\int_\Omega \varphi \,dx =-\int_\Omega M(x,\varphi)\, dx,\label{parphi}
 \ee
 Arguing as in \cite[Section 3.4]{EL} or \cite[Section 2.2]{GLS}, recalling (H3) and \eqref{inif}, we can infer that
 \be
 |\overline{\varphi}(t)|<1,\quad \forall\, t\geq 0.
 \label{averphie1}
 \ee

 Next, integrating \eqref{f2.b} over $\Omega$, we get
 \be
 \frac{d}{dt}\int_\Omega \sigma dx =\int_\Omega S(x,t)\, dx,\notag
 \ee
 so that
 \be
 \overline{\sigma}(t)=\overline{\sigma_{0}}+\int_0^t \overline{S(s)}ds,\quad \forall\, t\geq 0.
 \label{averphie}
 \ee
 Then it follows from the Poincar\'{e}--Wirtinger inequality that
 \be
 \|\sigma\|\leq C_P(\|\nabla \sigma\|+|\overline{\sigma}|)\leq C_P\|\nabla \sigma\|+C,
 \label{sigmaa}
 \ee
 where $C_P$ depends on $\Omega$ and $C$ depends on $\overline{\sigma_0}$, $S_0$, $\Omega$.
 
 Multiplying \eqref{f3.c} by $\bm{v}$, \eqref{f1.a} by $\mu$, \eqref{f4.d} by  $\partial_t \varphi$,\ \eqref{f2.b} by $\sigma+\chi(1-\varphi)$, integrating over $\Omega$ and adding the resultants together, we obtain (i.e., taking $\mathcal{C}=0$ in \eqref{BEL})
 \begin{align}
 	\frac{d}{dt}\mathcal{E}(t)
 	+\mathcal{D}(t)=-\int_\Omega M(x,\varphi)\mu \, dx +   \int_\Omega S(\sigma +\chi(1-\varphi))\, dx,
 	\label{BEL2}
 \end{align}
 where
 \begin{align}
 	\mathcal{E}(t)&=\frac{1}{2}\|\bm{v}(t)\|^2
 	+ \int_\Omega A\varPsi(\varphi(t))\, dx + \frac{B}{2}\|\nabla \varphi(t)\|^2
 	+\frac{1}{2}\|\sigma(t)\|^2\notag\\
 	&\quad
 	+\int_\Omega \chi\sigma(t)(1-\varphi(t))\, dx,\label{E}\\
 	\mathcal{D}(t)& =\int_\Omega 2\eta(\varphi(t))|D\bm{v}(t)|^2 + |\nabla \mu(t)|^2\, dx+\|\nabla(\sigma(t)+\chi(1-\varphi(t))\|^2.\label{D}
 \end{align}
 Next, using the Cauchy--Schwarz inequality, Young's inequality and \eqref{inif}, we easily see that
 \begin{align*}
 	\int_\Omega S(\sigma +\chi(1-\varphi)) dx
 	\leq \xi \|\sigma\|^2+\xi^{-1}\|S\|^2+C,
 \end{align*}
 where $\xi\in (0,1)$, the positive constant $C$ may depend on $\chi$, $\Omega$, but is independent of the initial data.
 Next, multiplying \eqref{f4.d} by $\Delta\varphi$ and integrating over $\Omega$, we get
 \begin{align}
 	&B\|\Delta\varphi\|^2 
 	=A \int_{\Omega}\varPsi'(\varphi)\Delta\varphi\, dx -\chi\int_\Omega \sigma\Delta\varphi\, dx-\int_\Omega\nabla \mu\cdot\nabla\varphi\,dx .
 	\label{energy e}
 \end{align}
 The three terms on the right-hand side of \eqref{energy e} can be estimated as follows:
 \begin{align}
 	A\int_{\Omega}\varPsi'(\varphi)\Delta\varphi \ dx
 	&=- A\int_{\Omega}\varPsi_{0}''(\varphi)|\nabla\varphi|^2 \, dx
 	-A\theta_{0}\int_{\Omega}\varphi\Delta\varphi\ dx\notag\\
 	&\le \frac{B}{8}\|\Delta\varphi\|^2  +\frac{2A^2\theta_0^2}{B}\|\varphi\|^2,\label{diss1}
 \end{align}
 \begin{align}
 	-\chi \int_{\Omega} \sigma \Delta\varphi dx
 	&=\chi\int_{\Omega}\nabla\varphi\cdot \nabla\sigma \, dx \leq \chi^2\|\nabla\varphi\|^2+\frac{1}{4}\|\nabla\sigma\|^2 \notag\\
 	& \le \frac{B}{8}\|\Delta\varphi\|^2  +\frac{1}{4}\|\nabla\sigma\|^2 +\frac{2\chi^4}{B}\|\varphi\|^2,
 	\label{diss2}
 \end{align}
 \begin{align}
 	- \int_{\Omega} \nabla\mu\cdot \nabla\varphi dx
 	& \leq \|\nabla\mu\|\|\nabla\varphi\| \notag\\
 	& \le \frac{1}{2}\|\nabla\mu\|^2+\frac{B}{8}\|\Delta\varphi\|^2   +\frac{1}{2B}\|\varphi\|^2.
 	\label{diss3}
 \end{align}
 In light of \eqref{inif} and \eqref{energy e}--\eqref{diss3}, we have
 \begin{align}
 	&\frac{B}{2}\|\Delta\varphi\|^2 +\|\nabla\varphi\|^2 \le \frac{1}{2}\|\nabla\mu\|^2+\frac{1}{4}\|\nabla\sigma\|^2+C.
 	\label{energy e0}
 \end{align}
Now multiplying \eqref{f2.b} by $\sigma$ and integrating over $\Omega$, we get
\begin{align}
	\frac12 \frac{d}{dt}\|\sigma\|^2
	+\|\nabla \sigma\|^2
	=-\chi\int_{\Omega} \sigma \Delta\varphi  \ dx +\int_\Omega S\sigma dx.
	\notag
\end{align}
In light of \eqref{diss2}, we get
\begin{align}
	\frac12 \frac{d}{dt}\|\sigma\|^2
	+\frac34\|\nabla \sigma\|^2
	\leq \frac{B}{4}\|\Delta\varphi\|^2 +\xi\|\sigma\|^2 +\xi^{-1}\|S\|^2 +C.
	\label{energy e3}
\end{align} 

Adding \eqref{energy e} with \eqref{energy e3} and using \eqref{sigmaa}, we see that
\begin{align}
	\frac12 \frac{d}{dt}\|\sigma\|^2
	+\frac{B}{2}\|\Delta\varphi\|^2 + \frac12\|\nabla \sigma\|^2 & \leq \xi\|\sigma\|^2 +\xi^{-1}\|S\|^2+C\notag\\
	&\leq \xi C_P^2\|\nabla \sigma\|^2+\xi^{-1}\|S\|^2+C(1+\xi^{-1}).
	\label{energy e5}
\end{align}
Recalling (H3) and \cite[Lemma 2.2]{L2022}, we have
\be
\int_{\Omega} \left|\varPsi_0^{\prime}(\varphi)\right| \mathrm{d} x \leq C(1+\int_{\Omega} \varPsi_0^{\prime}(\varphi) M(x, \varphi)\, \mathrm{d} x).
\ee
On account of the boundedness of $M$, the first on the right-hand side of \eqref{BEL2} can be estimated as follows:
\begin{align}
	-\int_\Omega M(x,\varphi)\mu \, dx
	&= B\int_\Omega M(x,\varphi)\Delta\varphi\,dx -  A\int_\Omega M(x,\varphi)\varPsi'(\varphi)\,dx + \chi\int_\Omega M(x,\varphi)\sigma\,dx\notag\\
	&\leq \frac{B}{8}\|\Delta \varphi\|^2
	-  A\int_\Omega M(x,\varphi)\varPsi_0'(\varphi)\,dx+A\theta_0\int_\Omega M(x,\varphi)\varphi\,dx\notag\\
	&\quad
	- \chi\int_\Omega M(x,\varphi)\sigma\,dx+C\notag\\
	& \leq \frac{B}{8}\|\Delta \varphi\|^2
	-\|\varPsi_0'(\varphi)\|_{L^1}+\xi\|\sigma\|^2+C.\label{mueea}
\end{align} 
Recalling Korn's inequality and Poincar\'{e}'s inequality, we get
\be
\|\bm{v}\| \le C_P'\|D\bm{v}\|, \label{po}
\ee
where the constant $C_P'$ only depends on $\Omega$. Noticing \eqref{inif}, we have
\begin{align}
	\mathcal{E}(t)&=\frac{1}{2}\|\bm{v}(t)\|^2
	+ \int_\Omega A\varPsi(\varphi(t))\, dx + \frac{B}{2}\|\nabla \varphi(t)\|^2
	+\frac{1}{2}\|\sigma(t)\|^2\notag\\
	&\quad
	+\int_\Omega \chi\sigma(t)(1-\varphi(t))\, dx\notag\\
	&\le C(\|D\bm{v}\|^2+\|\nabla\sigma\|^2+\|\nabla\varphi\|^2+1).\label{E1}
\end{align}
Using \eqref{mueea} and \eqref{E1}, we add \eqref{BEL2} with \eqref{energy e5}. We can first take $\xi\in(0,1)$ to be sufficiently small and then find a positive constant $\zeta\in (0,\xi)$ such that
\begin{align}
	&\frac{d}{dt}\widetilde{\mathcal{E}}(t)
	+\widetilde{\mathcal{D}}(t) + \zeta \widetilde{\mathcal{E}}(t)\leq  C(1+\|S\|^2),
	\label{BEL5}
\end{align}
where
\begin{align}
	&\widetilde{\mathcal{E}}(t)=\mathcal{E}(t)+\frac12 \|\sigma(t)\|^2+C_1,\notag \\
	&\widetilde{\mathcal{D}}(t)=\frac12 \mathcal{D}(t)+ \frac{B}{4}\|\Delta \varphi(t)\|^2+ \frac{1}{4}\|\nabla \sigma(t)\|^2+\|\varPsi'(\varphi)\|_{L^1}.\notag
\end{align}
The constants $\xi$, $\zeta$ as well as $C$ only depend on the coefficients of the system, $\Omega$, $\overline{\sigma_0}$, $S_0$. $C_1$ is a positive constant that guarantees $\widetilde{\mathcal{E}}(t)\geq 1$ and depends on $\varPsi$, $A$, $\chi$, $\Omega$, $\overline{\sigma_0}$, $S_0$, but not on other information of the initial data.
 
Applying a suitable Gronwall's lemma (see, e.g., \cite[Lemma 2.5]{GGP}) to the differential inequality \eqref{BEL5}, we deduce that for all $t \geq 0$:
\begin{equation}
	\begin{aligned}
		\widetilde{\mathcal{E}}(t)
		\leq \widetilde{\mathcal{E}}(0) \mathrm{e}^{-\zeta t} + C \int_0^te^{-\zeta(t-s)}(1+\|S(s)\|^2)\,ds \leq \widetilde{\mathcal{E}}(0) \mathrm{e}^{-\zeta t}+C,
		\label{bd22}
	\end{aligned}
\end{equation}
and
\begin{align}
	\int_{t}^{t+1} \widetilde{\mathcal{D}}(s) \, d s \leq \widetilde{\mathcal{E}}(0) \mathrm{e}^{-\zeta t}+C,
	\label{bd2a2}
\end{align}
where the constant $C$ depends on the coefficients of the system, $\Omega$,  $\overline{\sigma_0}$, $S_0$, but not on other information of the initial data.

Hence, in view of the definitions of $\widetilde{\mathcal{E}}(t)$ and $\widetilde{\mathcal{D}}(t)$, we can conclude from the dissipative estimates \eqref{bd22}, \eqref{bd2a2} that
\begin{align}
	&\|\bm{v}(t)\|^2 +\|\varphi(t)\|_{H^1}^2+\|\sigma(t)\|^2 \notag \\
	&\quad +\int_{t}^{t+1}\big(\|\nabla\bm{v}(s)\|^2 +\|\Delta\varphi(s)\|^2 + \|\nabla \mu(s)\|^2+ \|\nabla \sigma(s)\|^2+\|\varPsi'(\varphi(s))\|_{L^1}\big)\,  d s \notag \\
	& \le C,\quad \forall\, t\geq 0,
	\notag
\end{align}
where the constant $C$ is independent of $t$.
Recalling \eqref{f4.d}, we imply from the Poincar\'{e}--Wirtinger inequality that
\begin{equation}
	\left \|  \mu  \right \|_{H^1 }\le  C (1+\|  \nabla \mu\|).
	\label{2amuL2H1}
\end{equation}

Using the above uniform-in-time estimate and by the same argument as in \cite[Section 3.2]{H} (see also \cite{GGM2017,GGW} for the cases without coupling with $\sigma$), it is straightforward to check that for all $t \geq 0$,
\begin{align}
	&\int_{t}^{t+1}\big(\|\varphi(s)\|_{H^2}^4+\| \mu(s)\|^{2}_{H^1(\Omega)}+ \|\varphi(s)\|^2_{W^{2,q}}+\|\varPsi_0'(\varphi(s))\|^2_{L^q}\big) \, d s \leq C,  \notag\\
	&\int_{t}^{t+1}\big( \left\|\partial_{t} \varphi(s)\right\|_{(H^1)'}^{2} +\left\| \partial_{t} \sigma(s)\right\|_{(H^1)'}^2\big)\,ds \leq C,\notag
\end{align}
where $q\in [2,+\infty)$ and the positive constant $C$ is independent of $t$.
As a consequence, we arrive at our conclusions \eqref{uni1} and \eqref{uni2}. 

The proof of Lemma \ref{lm1} is complete.
\hfill$\square$ 
 
Then, we prove Theorem \ref{attr} on the existence of global weak solutions to problem \eqref{f3.c}--\eqref{ini0}.

\textbf{Step 1.} From Theorem \ref{2main}, the system \eqref{f3.c}--\eqref{ini0} generates a dynamical system $(\mathcal{X}_{m_1,m_2},\mathcal{S}(t))$, which is a strongly continuous semigroup,
$$
\mathcal{S}(t): \mathcal{X}_{m_1,m_2}\to \mathcal{X}_{m_1,m_2}, \quad t \geq 0,
$$
acting by the formula
$$
\mathcal{S}(t)(\bm{v}_0,\varphi_0,\sigma_0)=(\bm{v}(t),\varphi(t),\sigma(t)).
$$
Here $(\bm{v}(t),\varphi(t),\sigma(t))$ is the unique solution to system \eqref{f3.c}--\eqref{ini0} in two dimensions. The one-parameter family of maps $\mathcal{S}(t)$ on $\mathcal{X}_{m_1,m_2}$ satisfy:
\be
\begin{aligned}
	&\mathcal{S}(0)=\operatorname{Id}_{\mathcal{X}_{m_1,m_2}},\\
	&\mathcal{S}(t+\tau)=\mathcal{S}(t) \mathcal{S}(\tau),\quad\forall\,  t, \tau \geq 0,\\
	&t \mapsto \mathcal{S}(t)(\bm{v}_0,\varphi_0,\sigma_0) \in \mathcal{C}\left([0, \infty), \mathcal{X}_{m_1,m_2}\right),\quad\forall\,(\bm{v}_0,\varphi_0,\sigma_0)\in \mathcal{X}_{m_1,m_2}.
\end{aligned}
\ee

We consider $(\bm{v}_0,\varphi_0,\sigma_0)  \in \mathcal{B}\subset \mathcal{X}_{m_1,m_2}$ and $(\bm{v}(t),\varphi(t),\sigma(t))=S(t)(\bm{v}_0,\varphi_0,\sigma_0)$, for any $t \geq 0$.

It can be implied from the above estimates the dissipativity property of $\mathcal{S}(t)$ in $\mathcal{X}_{m_1,m_2}$, that is,
$\mathcal{S}(t)$ possesses a bounded absorbing ball
$$
\mathcal{B}_0 =\{ (\bm{v},\varphi,\sigma)\in\mathcal{X}_{m_1,m_2}\, :\, \|\bm{v}\|+|\varphi\|_{H^1}+\|\sigma\| \le R_0\},
$$
where the radius $R_0>0$ depends on coefficients of the system, $\Omega$, $m_1$, $m_2$ and but is independent of the initial data. For
every bounded set $\mathcal{B} \subset \mathcal{X}_{m_1,m_2}$, there exists a time $t_0 = t_0(\mathcal{B}) > 0$ satisfying
\begin{equation}
\mathcal{S}(t)\mathcal{B}\subset \mathcal{B}_0,\quad \forall\,t\geq t_0.\label{abso1}
\end{equation}
For any fixed constant $r > 0$, integrating \eqref{BEL5} on the time interval $[t, t + r]$, we obtain
\begin{align}
\int_{t}^{t+r}\big(\|\nabla\bm{v}(s)\|^2
+\|\varphi(s)\|_{H^2}^2 + \| \mu(s)\|^2_{H^1}+ \|\nabla \sigma(s)\|^2\big)\,  d s \le C,\quad \forall\,t \ge t_0(\mathcal{B}),\label{uni11}
\end{align}
where the constant $C$ depends on the coefficients of the system, $\Omega$,  $\overline{\sigma_0}$ and $r$.

From the same argument as in \cite[Section 3.2]{H} (see also \cite{GGM2017,GGW} for the cases without coupling with $\sigma$), it can be implied that
\begin{align}
\int_{t}^{t+r}\big( \left\|\partial_{t} \varphi(s)\right\|_{(H^1)'}^{2} +\left\| \partial_{t} \sigma(s)\right\|_{(H^1)'}^2\big)\,ds \leq C, \quad \forall\,t \ge t_0(\mathcal{B}).\label{uni23}
\end{align}

\textbf{Step 2.} 
Then we obtain the higher order estimates. Multiplying equation \eqref{f3.c} by $\bm{S} \bm{v}$ and integrating over $\Omega$, we obtain
\begin{equation}
	\frac{1}{2} \frac{d}{d t}\left\|\nabla\bm{v}\right\|^{2} +\left(\bm{v}\cdot\nabla\bm{v}, \bm{S}\bm{v}\right)-2\left(\operatorname{div}\left(\eta\left(\varphi\right) D\bm{v}\right), \bm{S}\bm{v}\right)=\left((\mu+\chi\sigma) \nabla \varphi, \bm{S}\bm{v}\right). \label{v3}
\end{equation}
Following the computations in the proof of \cite[Theorem 4.1]{GMT}, we obtain
\be
\begin{aligned}
	-2\left(\operatorname{div}\left(\eta\left(\varphi\right) D \bm{v}\right), \bm{S}\bm{v}\right)
	&=-\left(\eta\left(\varphi\right) \Delta \bm{v}, \bm{S}\bm{v}\right)-2\left(\eta^{\prime}\left(\varphi\right) D \bm{v} \nabla \varphi, \bm{S}\bm{v}\right)\\
	&\geq \frac{5\eta_{*}}{6}\left\|\bm{S}\bm{v}\right\|^{2} -C\big(1+\left\|\varphi\right\|_{H^{2}}^{4}\big)\left\|\nabla \bm{v}\right\|^{2},
\end{aligned}\notag
\ee
\begin{equation}
	\begin{aligned}
		|\left((\bm{v}\cdot\nabla)\bm{v}, \bm{S}\bm{v}\right)|
		& \leq\left\|\bm{v}\right\| _{\bm{L}^{4}}\left\|\nabla \bm{v}\right\| _{\bm{L}^{4}}\left\|\bm{S}\bm{v}\right\|
		\leq \frac{\eta_{*}}{6}\left\|\bm{S}\bm{v}\right\|^{2} +C\left\|\nabla \bm{v}\right\|^{4},
	\end{aligned}
	\notag
\end{equation}
where the positive constant $C$ may depend on $\eta_*$, $\Omega$ and coefficients of the system. Next, in light of Young's inequality and the Sobolev embedding theorem, we have
\begin{equation}
	\begin{aligned}
		\left((\mu+\chi\sigma) \nabla \varphi, \bm{S} \bm{v}\right)
		& \leq\left\|\mu+\chi\sigma\right\|_{L^{6}}\left\|\nabla \varphi\right\|_{\bm{L}^{3}} \left\|\bm{S} \bm{v}\right\| \\
		& \leq \frac{\eta_{*}}{6}\left\|\bm{S} \bm{v}\right\|^{2} +C\left\|\varphi\right\|_{H^{2}}^{2} \big(1+\left\|\nabla\mu\right\|^{2}+\|\sigma\|_{H^1}^2\big).
	\end{aligned}\notag
\end{equation}
Combining the above estimates, we deduce from \eqref{v3} that
\begin{align}
	\frac{1}{2} \frac{d}{d t}\left\|\nabla \bm{v}\right\|^{2}+\frac{\eta_{*}}{2}\left\|\bm{S} \bm{v}\right\|^{2}
	& \leq C\left\|\nabla \bm{v}\right\|^{4}+ C(1+\|\varphi\|_{H^{2}}^{4})\left\|\nabla \bm{v}\right\|^{2}\notag\\
	&\quad+C\|\varphi\|_{H^{2}}^{2}(1+\| \nabla\mu\|^{2}+\|\sigma\|_{H^1}^2).
	\label{sv}
\end{align}
Multiplying equation \eqref{f3.c} by $\partial_{t} \bm{v}$ and integrating over $\Omega$, we have
\begin{equation}
	\left\|\partial_{t}\bm{v}\right\|^{2} +\left(\bm{v}\cdot\nabla\bm{v}, \partial_{t}\bm{v}\right) -2\left(\operatorname{div}\left(\eta\left(\varphi\right) D\bm{v}\right), \partial_{t}\bm{v}\right)=\left((\mu+\chi\sigma) \nabla \varphi, \partial_{t}\bm{v}\right).
	\label{v1}
\end{equation}
By Young's inequality and the Sobolev embedding theorem, we have
\begin{align}
	\left((\mu+\chi\sigma) \nabla \varphi, \partial_{t} \bm{v}\right)
	&  \leq \left\|\mu+\chi\sigma\right\|_{L^{6}}\left\|\nabla \varphi\right\|_{\bm{L}^{3}}\left\|\partial_{t} \bm{v}\right\| \notag\\
	& \leq \frac{1}{6}\left\|\partial_{t} \bm{v}\right\|^{2}+C\left\|\varphi\right\|_{H^{2}}^{2} \big(1+\left\| \nabla\mu\right\|^{2}+\|\sigma\|_{H^1}^2\big). \notag
\end{align}
The other two terms on the left hand side of \eqref{v1} can be estimated exactly as in the proof of \cite[Theorem 4.1]{GMT} such that
\begin{equation}
	\begin{aligned}
		|\left(\bm{v}\cdot\nabla\bm{v}, \partial_{t}\bm{v}\right) | &  \leq \frac{1}{6}\left\|\partial_{t}\bm{v}\right\|^{2} +C\big(\left\|\bm{S}\bm{v}\right\|^{2} +\left\|\nabla\bm{v}\right\|^{4}\big),\\
		2\left(\operatorname{div}\left(\eta\left(\varphi\right) D \bm{v}\right), \partial_{t} \bm{v}\right)
		&\leq \frac{1}{6}\left\|\partial_{t} \bm{v}\right\|^{2}+C\big(\left\|\bm{S} \bm{v}\right\|^{2}+\left\|\varphi\right\|_{H^{2}}^{2}\left\|\nabla \bm{v}\right\|^{2}\big).
	\end{aligned}\notag
\end{equation}
As a consequence,
\be
\begin{aligned}
	\left\|\partial_{t} \bm{v}\right\|^{2} &\leq C_2 \left\|\bm{S} \bm{v}\right\|^{2}+ C\left\|\nabla \bm{v}\right\|^{4} + C\left\|\varphi\right\|_{H^{2}}^{2} \big(1+\left\|\nabla \bm{v}\right\|^{2}+\left\| \nabla\mu\right\|^{2}+\|\sigma\|_{H^1}^2\big).\label{vt}
\end{aligned}
\ee
Multiplying \eqref{vt} by $\varpi=\frac{\eta_{*}}{4 C_{2}}>0$ and adding the resultant with \eqref{sv}, we find
\begin{equation}
	\begin{aligned}
		&\frac{1}{2} \frac{d}{d t}\|\nabla \bm{v}\|^{2} +\frac{\eta_{*}}{4}\|\bm{S} \bm{v}\|^{2} +\varpi\|\partial_{t} \bm{v}\|^{2}\\
		& \quad \leq C\|\nabla \bm{v}\|^{4}+ C\big(1+\|\varphi\|_{H^{2}}^{4})(1+\|\nabla \bm{v}\|^{2}+\| \nabla\mu\|^{2}+\|\sigma\|_{H^1}^2\big).
		\label{2v1}
	\end{aligned}
\end{equation}
\textbf{Estimate for $\|\nabla \mu\|$}.
Multiplying equation \eqref{f1.a} by $\partial_{t} \mu$ and integrating over $\Omega$, we obtain
\be
\begin{aligned}
	 &\frac{d}{d t}\left(\frac12\int_{\Omega}|\nabla \mu|^{2}\,dx\right)+\left( \partial_{t} \varphi,\partial_{t} \mu\right)+\left( \bm{v} \cdot \nabla \varphi,\partial_{t} \mu\right)+(M(x,\varphi),\partial_{t} \mu)=0.\label{mu}
\end{aligned}
\ee
As in \cite{MT}, we write
\be
\begin{aligned}\left( \bm{v} \cdot \nabla \varphi,\partial_{t} \mu\right) = \frac{d}{d t} \left(\bm{v} \cdot \nabla \varphi, \mu\right)  -\left(\partial_{t} \bm{v}\cdot \nabla \varphi, \mu\right)-\left(\bm{v} \cdot \nabla \partial_{t} \varphi, \mu\right). \end{aligned}\label{d1}
\ee
Then by Young's inequality, the Sobolev embedding theorem and \eqref{2amuL2H1}, we get
\begin{align}
	(\partial_{t}\bm{v} \cdot \nabla \varphi, \mu) &\le\left\|\nabla \varphi\right\|_{\bm{L}^{3}}\left\|\partial_{t} \bm{v}\right\|\left\|\mu\right\|_{L^{6}} \leq \frac{\varpi}{4}\left\|\partial_{t} \bm{v}\right\|^{2} +C \left\|\varphi\right\|_{H^{2}}^{2} \big(1+\left\|\nabla \mu\right\|^{2} +\|\sigma\|_{H^1}^2\big),\nonumber
\end{align}
and
\begin{equation}
	\begin{aligned}\left( \bm{v}\cdot \nabla \partial_{t} \varphi,\mu\right) & \leq\left\|\bm{v}\right\|_{ \bm{L}^3}\left\|\nabla \partial_{t} \varphi\right\|\left\|\mu\right\|_{L^{6}} \leq \frac{B}{4}\left\|\nabla \partial_{t} \varphi\right\|^{2}+C \left\|\nabla \bm{v}\right\|^2\big(1+\left\|\nabla \mu\right\|^{2}+\|\sigma\|_{H^1}^2\big).
	\end{aligned}\notag
\end{equation}
Using the Poincar\'{e}--Wirtinger inequality and recalling \eqref{inif}, we deduce that

\begin{align}
	|A(\theta-\theta_0)|\left\|\partial_{t} \varphi\right\|^{2}&=	|A(\theta-\theta_0)|(\Delta \mu-M(x,\varphi)-\bm{v} \cdot \nabla \varphi,\partial_{t} \varphi)\notag\\
	&=|A(\theta-\theta_0)|(\nabla \mu+\bm{v}  \varphi,\nabla\partial_{t} \varphi)+|A(\theta-\theta_0)|(-M(x,\varphi),\partial_{t} \varphi)\notag\\
	&\leq\frac{B}{16} \left\|\nabla \partial_{t} \varphi\right\|^{2}+\frac12|A(\theta-\theta_0)|\left\|\partial_{t} \varphi\right\|^{2} +C \big(1+\left\|\nabla\mu \right\|^2 +\left\| M(x,\varphi)\right\|^2+\|\bm{v}\|^2\big),\notag
\end{align}
thus we have
\begin{align}
	|A(\theta-\theta_0)|\left\|\partial_{t} \varphi\right\|^{2}
	&\leq\frac{B}{8} \left\|\nabla \partial_{t} \varphi\right\|^{2} +C \big(1+\left\|\nabla\mu \right\|^2 +\left\| M(x,\varphi)\right\|^2+\|\bm{v}\|^2\big),\notag
\end{align}
and
\begin{align}
	|\chi\left(  \partial_{t} \varphi,\partial_{t} \sigma\right)|
	&\leq C\left\| \partial_{t} \varphi\right\|_{H^1}\left\| \partial_{t} \sigma\right\|_{(H^1)'}\notag\\
	&\leq C(\left\| \partial_{t} \varphi\right\|+\left\|\nabla \partial_{t} \varphi\right\|)\left\| \partial_{t} \sigma\right\|_{(H^1)'}\notag\\
	&\leq\frac{B}{8} \left\|\nabla \partial_{t} \varphi\right\|^{2} +C \big(1+\left\|\nabla\mu \right\|^2 +\left\| M(x,\varphi)\right\|^2+\|\bm{v}\|^2+\left\| \partial_{t} \sigma\right\|^2_{(H^1)'}\big).
	\notag
\end{align}
Finally, from the assumption  (H2) and the above two estimates, we obtain
\begin{align}\left( \partial_{t} \varphi,\partial_{t} \mu\right)
	&=B\left\|\nabla \partial_{t} \varphi\right\|^{2} +A\left(\partial_{t} \varphi,\varPsi''\left(\varphi\right) \partial_{t} \varphi\right) -\chi\left(  \partial_{t} \varphi,\partial_{t} \sigma\right) \notag\\
	& \geq B\left\|\nabla \partial_{t} \varphi\right\|^{2}-|A(\theta-\theta_0)|\left\|\partial_{t} \varphi\right\|^{2} -|\chi(\partial_{t} \varphi,\, \partial_{t} \sigma)|\notag\\
	& \geq \frac{3B}{4}\left\|\nabla \partial_{t} \varphi\right\|^{2} -C \big(1+\left\|\nabla\mu \right\|^2 +\left\| M(x,\varphi)\right\|^2+\|\bm{v}\|^2+\left\| \partial_{t} \sigma\right\|^2_{(H^1)'}\big).
	\notag
\end{align}
Combining the above estimates, we deduce from  \eqref{mu} that
\begin{align}
	&\frac{d}{d t}\Big(\frac{1}{2}\|\nabla \mu\|^{2}+(\bm{v} \cdot \nabla \varphi, \mu)\Big) +\frac{B}{2}\|\nabla \partial_{t} \varphi\|^{2} \notag\\
	&\quad \leq \frac{\varpi}{4}\left\|\partial_{t} \bm{v}\right\|^{2} + C \big(1+\left\|\nabla\mu \right\|^2 +\left\| M(x,\varphi)\right\|^2+\|\bm{v}\|^2+\left\| \partial_{t} \sigma\right\|^2_{(H^1)'}\big) \notag\\
	&\qquad  +C \big(1+\|\varphi\|_{H^{2}}^{2}+\|\nabla\boldsymbol{v}\|^{2}\big) \big(1+\|\nabla \mu\|^{2}+\|\sigma\|_{H^1}^2\big).
	\label{22mu}
\end{align}
\textbf{Estimate for $\|\nabla \sigma\|$}. Multiplying \eqref{f2.b} by $-\Delta\sigma$ and integrating over $\Omega$, we get
\be
\frac{1}{2} \frac{d}{d t}\left\|\nabla\sigma\right\|^{2}-(\bm{v} \cdot \nabla \sigma,\Delta\sigma)+\|\Delta \sigma\|^2= \chi (\Delta \varphi,\Delta\sigma) - (S,\Delta\sigma ).
\ee
By the Gagliardo--Nirenberg inequality and Young's inequalities, we see that
\begin{align}
	|(\bm{v} \cdot \nabla \sigma,\Delta\sigma)|
	&\le \|\bm{v}\|_{\bm{L}^{4}}\| \nabla\sigma\|_{\bm{L}^{4}}\|\Delta\sigma\|\notag\\
	&\le\frac{1}{8}\|\Delta\sigma\|^2 + C\|\bm{S}\bm{v}\|^{\frac{1}{2}} \|\bm{v}\|^{\frac{3}{2}}\|\sigma\|_{H^2} \|\sigma\|_{H^1}\notag\\
	&\leq \frac{1}{8}\|\Delta\sigma\|^2 + C\|\bm{S}\bm{v}\|^{\frac{1}{2}} \|\bm{v}\|^{\frac{3}{2}}(\|\Delta \sigma\|+\|\sigma\|)\|\sigma\|_{H^1} \notag\\
	&\le\frac{1}{4}\|\Delta\sigma\|^2 +\frac{\eta_*}{8}\|\bm{S}\bm{v}\|^{2} +
	C\big(\|\bm{v}\|^{6}\|\sigma\|_{H^1}^{4} +\|\bm{v}\|^{6}+\|\sigma\|_{H^1}^{4}\big).\notag
\end{align}
The remaining terms can be easily handled by using Young's inequality
\begin{align*}
	|\chi (\Delta \varphi,\Delta\sigma) - (S,\Delta\sigma ) | \leq \frac14\|\Delta \sigma\|^2+C\big(\|\Delta\varphi\|^2+\|\sigma\|^2+\|S\|^2\big).
\end{align*}
Hence, we have
\begin{align}
	&\frac{1}{2} \frac{d}{d t}\left\|\nabla\sigma\right\|^{2}+\frac{1}{2}\|\Delta \sigma\|^2 \notag\\
	&\quad \leq \frac{\eta_*}{8}\|\bm{S}\bm{v}\|^{2}+ C \big(\left\|\Delta\varphi\right\|^{2} + \|\bm{v}\|^{6}\|\sigma\|_{H^1}^{4} +\|\bm{v}\|^{6} +\|\sigma\|_{H^1}^{4}+\left\|S\right\|^{2}\big).
	\label{2sig}
\end{align}


Recalling the higher order estimates in \cite[Section 3.3]{H1}, there exists a positive constant $\varpi$ depending coefficients of the system and $\Omega$ such that
\be
\frac{d}{d t}\Lambda_1(t)  +  \mathcal{M}(t)  \leq    \mathcal{R}_{1}(t)\Lambda_2(t)+\mathcal{R}_{2}(t),\quad   \forall\,t\in[t_0(\mathcal{B}),+\infty).
\label{2menergy}
\ee
with
\begin{align*}
&\Lambda_1(t)=\frac12\|\nabla \bm{v}\|^{2}+\frac{1}{2}\|\nabla \mu\|^{2}+(\bm{v} \cdot \nabla \varphi, \mu)+\frac{1}{2}\left\|\nabla\sigma\right\|^{2},
\notag\\
&\mathcal{M}(t) =\frac{\eta_{*}}{8}\|\bm{S} \bm{v}\|^{2}+\frac{3\varpi}{4}\|\partial_{t} \bm{v}\|^{2} +\frac{B}{2}\|\nabla \partial_{t} \varphi\|^{2} +\frac{1}{2}\|\Delta \sigma\|^2,\\
&\mathcal{R}_{1}(t) = C\big(1+\|\nabla \bm{v}\|^{2}+\|\varphi\|_{H^{2}}^{4}+\|\nabla \sigma\|^2\big),\\
&\Lambda_2(t)=1+\|\nabla \bm{v}\|^{2}+\|\nabla \mu\|^{2}+\|\nabla\sigma\|^2,\\
&\mathcal{R}_{2}(t) =C\big(1+\left\|\nabla\mu \right\|^2 +\left\| M(x,\varphi)\right\|^2+\|\bm{v}\|^2+\left\| \partial_{t} \sigma\right\|_{(H^1)'}^2 + \left\|\Delta  \varphi\right\|^{2} +\left\|S\right\|^{2}\big),
\end{align*}
where the constant $C$ depends on the coefficients of the system, $\Omega$,  $\overline{\sigma_0}$, $S_0$. Recalling
 $$\sup _{0 \leq t <+\infty}\left\|\varphi(t)\right\|_{L^{\infty}} \leqslant 1,$$ 
we imply that
\be\begin{aligned}
	|(\bm{v} \cdot \nabla \varphi, \mu)|
	&=|(\bm{v}, \varphi \nabla  \mu)| \leq 2\left\|\bm{v}\right\|^2+\frac{1}{4}\left\|\nabla \mu\right\|^{2}.
\end{aligned}
\ee
 Therefore, it can be implied from the standard Gronwall lemma yields a local estimate for $\Lambda_1(t)$ on $[t_0(\mathcal{B}),t_0(\mathcal{B})+1]$. From the uniform Gronwall lemma (see \cite[Chapter III, Lemma 1.1]{T}), \eqref{uni1} and \eqref{uni2}, we deduce that $\Lambda_1(t)\leq C$ for all $t\geq t_0(\mathcal{B})+1$. By these two estimates we obtain that $\Lambda_1(t)$ is uniformly bounded on the time interval $[t_0(\mathcal{B})+1,+\infty)$. Thus, we have
\begin{align}
\|\bm{v}(t)\|_{\bm{H}^1}+\|\mu(t)\|_{H^1}+\|\sigma(t)\|_{H^1}\leq C_9,\label{ener}\\
\int_t^{t+1}\| \bm{v}(s)\|^{2}_{\bm{H}^2}+\|\partial_{t} \bm{v}(s)\|^{2} +\| \partial_{t} \varphi(s)\|^{2}_{H^1} +\| \sigma(s)\|^2_{H^2}\,ds\le C, \quad \forall\, t\geq t_0(\mathcal{B})+1,\label{ener1}
\end{align}
where the positive constant $C$ is independent of $t$, may depend on the coefficients of the system, $\Omega$,  $\overline{\sigma_0}$, $S_0$.

\textbf{Separation property}. We refer to \cite{A2009,CG,GGM2017,GGW,GMT} for the following elliptic estimate that for any $q\in [2,+\infty)$, there exists a positive constant $C=C(q,A,B,\Omega)$ satisfying
\be
\begin{aligned}
	&\|\varphi\|_{W^{2, q}}+\left\|\varPsi_0^{\prime}(\varphi)\right\|_{L^q} \leq C\left(\|\mu\|_{H^1}+\|\varphi\|_{H^1}+\|\sigma\|_{H^1}\right) ,\\
	&\left\|\varPsi_0^{\prime \prime}(\varphi)\right\|_{L^p} \leq C\left(1+e^{C \| \mu+\theta_0 \varphi+\chi\sigma\|_{H^1} ^2}\right).
\end{aligned}
\ee
Thus, using \eqref{ener}, for any $q\in [2,+\infty)$, there exists $C$ such that
\be
\begin{aligned}
	&\|\varphi(t)\|_{W^{2, q}}+\left\|\varPsi_0^{\prime}(\varphi(t))\right\|_{L^q} \leq C, \quad \forall\, t \in [t_0(\mathcal{B})+1,+\infty) ,\\
	&\left\|\varPsi_0^{\prime \prime}(\varphi(t))\right\|_{L^p}\le C,\quad \forall\, t \in [t_0(\mathcal{B})+1,+\infty) .
\end{aligned}
\ee
As $\partial_{x_i} \varPsi_0^{\prime}(\varphi)=\varPsi_0^{\prime \prime}(\varphi) \partial_{x_i} \varphi$ in the distributional sense (see, e.g., \cite{H1}), we obtain
$$
\left\|\varPsi_0^{\prime}(\varphi(t))\right\|_{W^{1,3}(\Omega)} \leq C, \quad \forall t\in [t_0(\mathcal{B})+1,+\infty).
$$
Thanks to \cite[Lemma 3.2]{H1}, there exists a constant $\widehat{\delta}\in (0,1)$ satisfying the uniform separation property holds:
\be
\|\varphi(t)\|_{\mathcal{C}(\bar{\Omega})} \leq 1-\widehat{\delta}, \quad \forall t\in [t_0(\mathcal{B})+1,+\infty).\label{sepa}
\ee
Finally, as $-B \Delta \varphi=\mu-A\varPsi^{\prime}(\varphi)+\chi\sigma$, we infer from \eqref{ener} that there exists a constant $C$ independent of the initial condition satisfying
\be
\|\varphi(t)\|_{H^3} \leq C, \quad \forall t \geq [t_0(\mathcal{B})+1,+\infty).\label{h3}
\ee
Finally, a comparison in the equation for $\sigma$ yields
\begin{align}
	\int_t^{t+1} \|\partial_t\sigma(s)\|^2\,ds
	&\leq C\int_t^{t+1} \big(\|\Delta\sigma(s)\|^2
	+\|\Delta\varphi(s)\|^2 + \|\bm{v}(s)\|_{\bm{L}^6}^2 \|\nabla\sigma(s)\|_{\bm{L}^3}^2\big)\,ds\notag \\
	&\leq C\int_t^{t+1}\big(\|\sigma(s)\|_{H^2}^2
	+\|\varphi(s)\|_{H^2}^2+ \|\bm{v}(s)\|_{\bm{H}^1}^2 \|\sigma(s)\|_{H^2}^2\big)\,ds\notag\\
	&\leq C.\label{vsigmt2d}
\end{align}
 \textbf{Higher-order estimates}.
Applying
$$
\partial_t^h u:=\frac{1}{h}(u(t+h)-u(t)), \quad \text { for } h>0,
$$   
 to \eqref{f1.a}, we have  
   \be 
   \partial_t \partial_t^h\varphi+\partial_t^h\bm{v} \cdot \nabla \varphi+\bm{v} \cdot \nabla \partial_t^h\varphi=\Delta\partial_t^h \mu-\partial_t^hM(x,\varphi).\label{phit}
   \ee
   Testing \eqref{phit} with $\partial_t^h\varphi$, we arrive at
   \begin{equation}
   	\begin{aligned}
   		\frac{d}{d t} & \frac{1}{2}\|\partial_t^h \varphi\|^2+B\|\Delta \partial_t^h \varphi\|^2 \\
   		& =\frac{A}{h}\left(\varPsi_0^{\prime}(\varphi(t+h))-\varPsi_0^{\prime}(\varphi(t)), \Delta \partial_t^h \varphi\right)+\theta_0\|\nabla \partial_t^h \varphi\|^2-\left(\partial_t^h M(\varphi), \partial_t^h \varphi\right) \\
   		&\quad-\left(\partial_t^h\bm{v} \cdot \nabla \varphi, \partial_t^h \varphi\right)-\left(\bm{v} \cdot \nabla \partial_t^h\varphi, \partial_t^h \varphi\right)-\left(\chi\Delta\partial_t^h\sigma, \partial_t^h \varphi\right).\label{phit-b}
   	\end{aligned}
   \end{equation}
  We notice that
  
  \begin{equation}
  	\begin{aligned}
  		& \frac{A}{h}\left(\varPsi_0^{\prime}(\varphi(t+h))-\varPsi_0^{\prime}(\varphi(t)), \Delta \partial_t^h \varphi\right) \\
  		& \quad=A\left(\int_0^1 \varPsi_0^{\prime \prime}(z \varphi(t+h)+(1-z) \varphi(t)) \mathrm{d} z \partial_t^h \varphi, \Delta \partial_t^h u\right) \\
  		& \quad \leq A\int_{\Omega} \sup_{s\in[-1+\delta,1-\delta]}\varPsi_0^{\prime \prime}(s)|\partial_t^h \varphi | | \Delta \partial_t^h \varphi| \,d x \\
  		& \quad \leq C\|\partial_t^h \varphi\|^2+\frac{B}{4}\|\Delta \partial_t^h \varphi\|^2 ,
  	\end{aligned}
  \end{equation}
  and
  \be
  -\left(\partial_t^h M(\varphi), \partial_t^h \varphi\right)\le C\|\partial_t^h \varphi\|^2.
  \ee
  It can be implied that
  \be
  -\left(\chi\Delta\partial_t^h\sigma, \partial_t^h \varphi\right)\le C\|\partial_t^h \sigma\|^2+\frac{B}{4}\|\Delta \partial_t^h \varphi\|^2,
  \ee
  \be
  -\left(\partial_t^h\bm{v} \cdot \nabla \varphi, \partial_t^h \varphi\right)=\left(\partial_t^h\bm{v}  \varphi, \nabla\partial_t^h \varphi\right)\le C\|\partial_t^h\bm{v}\|^2+C\|\nabla\partial_t^h \varphi\|^2,
  \ee
  and
  \be-\left(\bm{v} \cdot \nabla \partial_t^h\varphi, \partial_t^h \varphi\right)=0.\label{phit-e}
  \ee
  From \eqref{phit-b}--\eqref{phit-e}, we observe that
  \be 
  	\frac{d}{d t}  \frac{1}{2}\|\partial_t^h \varphi\|^2+B\|\Delta \partial_t^h \varphi\|^2\le C\|\partial_t^h \varphi\|^2+C(\|\partial_t^h\bm{v}\|^2+\|\nabla\partial_t^h \varphi\|^2+\|\partial_t^h \sigma\|^2).
  \ee
   Applying the uniform Gronwall lemma (see \cite[Chapter III, Lemma 1.1]{T}) and passing to the limit $h \to0$, it can be implied from elliptic regularity that
   
  \begin{align}
  	\|\partial_{t} \varphi(t)\|^{2}+\int_t^{t+1} \| \partial_{t} \varphi(s)\|^{2}_{H^2} \,ds\le C, \quad \forall\, t\geq t_0(\mathcal{B})+2,
  \end{align}  
  
  From \eqref{f1.a}, we observe that
  $$
  \|\Delta \mu\|\leq \|\partial_t\varphi\|+\|\bm{v}\|_{\bm{L}^3}\|\nabla\varphi\|_{\bm{L}^6} +\|M(x,\varphi)\|,
  $$
   thus we can imply from the Sobolev embedding theorem that
  \begin{align}
  	&\|\mu(t)\|_{H^2}\leq C,\quad \forall\, t\geq t_0(\mathcal{B})+2,\label{v}
  \end{align}
   
   \textbf{Step 3.} Let $\left\{\left(\boldsymbol{v}_{0, n}, \varphi_{0, n},\sigma_{0,n}\right)\right\}_{n \in \mathbb{N}}$ be a sequence in $\mathcal{X}_{m_1,m_2}$ and $\left(\boldsymbol{v}_0, \varphi_0,\sigma_0\right) \in \mathcal{X}_{m_1,m_2}$ such that $\left(\boldsymbol{v}_{0, n}, \varphi_{0, n},\sigma_{0,n}\right) \rightarrow\left(\boldsymbol{v}_0, \varphi_0,\sigma_0\right)$ in $\mathcal{X}_{m_1,m_2}$.   Assume that $\left(\boldsymbol{v}_n(t), \varphi_n(t),\sigma_n(t)\right)=S(t)\left(\boldsymbol{v}_{0, n}, \varphi_{0, n},\sigma_{0,n}\right)$ and $(\boldsymbol{v}, \varphi,\sigma)(t)=S(t)\left(\boldsymbol{v}_0, \varphi_0,\sigma_0\right)$. Let  $t^*>0$ be an any fixed constant. From the assumption, there exists a positive constant $M_0$ satisfying $\left\|\boldsymbol{v}_{0 ,n}\right\|+\left\|\varphi_{0, n}\right\|_{H^1}+\left\|\sigma_{0, n}\right\| \leq M_0$ and $\left\|\boldsymbol{v}_0\right\|+\left\|\varphi_0\right\|_{H^1}+\left\|\sigma_0\right\| \leq M_0$. Then we observe that
   \be
   \int_{\frac{t}{2}}^{\frac{3t}{4}}\left\|\boldsymbol{v}_n(s)\right\|_{\bm{H}^1}+\left\|\varphi_n(s)\right\|_{H^2}+\left\|\sigma_n(s)\right\|_{H^1}\,ds \le C,
   \ee
   where the positive constant $C$ may depend on the coefficients of the system, $\Omega$, $\overline{\sigma_0}$, $S_0$ and $M_0$, and $t$. Thus, there exists $\tau\in(\frac{t}{2},\frac{3t}{4})$ such that
   $$
   \|\bm{v}(\tau)\|_{\bm{H}^1}+\|\varphi(\tau)\|_{H^2}+\|\sigma(\tau)\|_{H^1}\leq C,
   $$
    where the positive constant $C$ may depend on the coefficients of the system, $\Omega$, $\overline{\sigma_0}$, $S_0$ and $M_0$, and $t$.
    Taking $(\bm{v}(\tau),\varphi(\tau),\sigma(\tau))$ as the initial data and recalling \eqref{2menergy}, there exists $M_1$ depending only on depend on the coefficients of the system, $\Omega$, $\overline{\sigma_0}$, $S_0$ and $M_0$, and $t$ such that $\left\|\boldsymbol{v}_n(t)\right\|_{\bm{H}^1}+\left\|\varphi_n(t)\right\|_{H^2}+\left\|\sigma_n(t)\right\|_{H^1} \leq M_1$ and $\|\boldsymbol{v}(t)\|_{\bm{H}^1}+\|\varphi(t)\|_{H^2}+\left\|\sigma(t)\right\|_{H^1} \leq M_1$. From \eqref{uni}, we imply that there exists $M_2$ depending only on the coefficients of the system, $\Omega$, $\overline{\sigma_0}$, $S_0$ and $M_0$ satisfying
\be
W(t)\le M_2\left(\frac{W(0)}{M_2}\right )^{\exp\big(-M_2\int_{0}^{t}Z(s)\, ds\big)},\quad \forall\, t\in [0,T],\label{uni4}
\ee
where
\begin{align*}
W(t)&= \frac{1}{2}\|\nabla\bm{S}^{-1}[\bm{v}(t)-\bm{v}_n(t)]\|^2+\frac{1}{2}\|\varphi(t)-\varphi_n(t) \|_{(H^1)'}^2  +\frac{1}{2}\|\sigma(t)-\sigma_n(t) \|_{(H^1)'}^2\\
&\quad +|\overline{\varphi}(t)-\overline{\varphi}_n(t)|,\\
Z(t)&= \|\nabla \bm{v}(t)\|^2 +\|\nabla \bm{v}_{n}(t)\|^2 +\|\varphi(t)\|_{W^{2,3}}^2 +\|\varphi_{n}(t)\|_{W^{2,3}}^2 +\|\varphi(t)\|_{H^{2}}^4\notag\\
&\quad +\|\varPsi'(\varphi)\|_{L^1}+\|\varPsi'(\varphi_{n})\|_{L^1}+\|\sigma_n(t)\|_{H^1}^2+1,
\end{align*}
 Then we notice that $\int_0^{t^*} Z(s) \mathrm{d} s \leq M_3$, where $M_3$ depends on $M_0, t^*$ and the coefficients of the system, but is independent of $n$. Taking $n$ sufficiently large such that $\|\nabla\bm{S}^{-1}[\bm{v}(t)-\bm{v}_n(t)]\|^2+\|\varphi(t)-\varphi_n(t) \|_{(H^1)'}^2  +\|\sigma(t)-\sigma_n(t) \|_{(H^1)'}^2<1$,  we deduce from a standard interpolation argument (cf. e.g., \cite{GT}) that
\be
\begin{aligned}
&\|\bm{v}(t)-\bm{v}_n(t)\|+\|\varphi(t)-\varphi_n(t) \|_{H^1} +\|\sigma(t)-\sigma_n(t) \| \\
&\quad\leq C\left(\|\nabla\bm{S}^{-1}[\bm{v}(t)-\bm{v}_n(t)]\|^{\frac12} +\|\varphi(t)-\varphi_n(t) \|_{(H^1)'}^{\frac14}  +\|\sigma(t)-\sigma_n(t) \|_{(H^1)'}^{\frac12}\right) \\
&\qquad \times\left(\|\bm{v}(t)-\bm{v}_n(t)\|_{\bm{H}^1}^{\frac12} +\|\varphi(t)-\varphi_n(t) \|_{H^2}^{\frac34}  +\|\sigma(t)-\sigma_n(t) \|_{H^1}^{\frac12}\right) \\
&\quad\leq C\left(2M_2^{\frac{1}{4}}+M_2^{\frac{1}{8}}\right)\left(2M_1^{\frac{1}{2}}+M_1^{\frac{3}{4}}\right)\\
&\qquad\times\left(\frac{\|\nabla\bm{S}^{-1}[\bm{v}(t)-\bm{v}_n(t)]\|^2+\|\varphi(t)-\varphi_n(t) \|_{(H^1)'}^2  +\|\sigma(t)-\sigma_n(t) \|_{(H^1)'}^2}{M_2}\right)^{\frac{1}{4} \mathrm{e}^{-M_4}},\label{conti}
\end{aligned}
\ee
where the constant $C$ is independent of $n$. Hence, we obtain that $\mathcal{S}(t)\in C(\mathcal{X}_{m_1,m_2}, \mathcal{X}_{m_1,m_2})$ for every $t\geq 0$. As a consequence, $\mathcal{S}(t)$ is actually a strongly continuous
semigroup on $\mathcal{X}_{m_1,m_2}$.

\textbf{Step 4.} Due to the above facts, we apply the abstract theory of infinite dimensional dynamical systems (see e.g., \cite[Chapter 1, Theorem 1.1]{T}) and conclude that there exists a global attractor $\mathcal{A}_{m_1,m_2}\subset \mathcal{X}_{m_1,m_2}$ for the dynamical system  $(\mathcal{S}(t),\mathcal{X}_{m_1,m_2})$ defined by problem \eqref{f1.a}--\eqref{ini0}. $\mathcal{A}_{m_1,m_2}$ is bounded in $\bm{H}^1(\Omega)\times H^2(\Omega)\times H^1(\Omega)$.

\textbf{Step 5.} We now prove that the global attractor fulfills further regularity  properties. 
Thus for any $(\bm{v},\varphi,\sigma)\in\mathcal{A}_{m_1,m_2}$, we obtain
   \be
   \|\varphi\|_{H^3}\le C,\quad\|\varphi\|_{\mathcal{C}(\bar{\Omega})} \leq 1-\widehat{\delta}.\label{sep11}
   \ee
The proof of Theorem \ref{attr} is complete.
\hfill$\square$
\section{Exponential attractor} \label{phs}
\setcounter{equation}{0}
\textbf{Proof of Theorem \ref{eattr}}. The separation property \eqref{sep11} plays a significant role in the proof of the exponential attractor, as we can overcome the singularity at $\pm1$ of the derivatives of the potential $\varPsi$. We  adapt the general framework in \cite{EFNT} (see also \cite{EMZ00,MZ08}) and employ the abstract conclusions in \cite{BG,GGMP06} to prove Theorem \ref{eattr}. \medskip

\textbf{Step 1}. By deducing an improved dissipativity of the semigroup $\mathcal{S}(t)$, we prove that there exists a compact absorbing set in $\mathcal{X}_{m_1,m_2}$.
\begin{lemma}\label{compatt}
	{\color{black}{There exist $R_1>0$ and $\delta_1\in (0,1)$ such that the ball
			\be
			\begin{aligned}
			\mathcal{B}_1 =&\{ (\bm{v},\varphi,\sigma)\in\mathcal{X}_{m_1,m_2}\cap (\bm{H}^1(\Omega)\times H^2(\Omega)\times H^1(\Omega)):\\
			&\quad\|\bm{v}\|_{\bm{H}^1(\Omega)}+\|\varphi\|_{H^2(\Omega)}+\|\sigma\|_{H^1(\Omega)}\le R_1,\ \|\varphi\|_{L^{\infty}(\Omega)}\le 1-\delta_1\}\label{me}
			\end{aligned}
			\ee
			is a compact absorbing set for $\mathcal{S}(t)$ in $\mathcal{X}_{m_1,m_2}$, that is, for
			every bounded set $\mathcal{B} \subset \mathcal{X}_{m_1,m_2}$, there exists a time $t_1 = t_1(\mathcal{B}) > 0$ satisfying $\mathcal{S}(t)\mathcal{B}\subset \mathcal{B}_1$, for all $t\geq t_1$.
	}}
\end{lemma}
\begin{proof}  {\color{black}{
			As we have proved that for
			every bounded set $\mathcal{B} \subset \mathcal{X}_{m_1,m_2}$ there exists an absorbing set $\mathcal{B}_0\subset \mathcal{X}_{m_1,m_2}$ (see \eqref{abso1}), we now need to investigate the evolution after the time $t_0=t_0(\mathcal{B})$, namely, starting from the set $\mathcal{B}_0$. For any initial data $(\bm{v}_0,\varphi_0,\sigma_0)\in \mathcal{B}_0$, with $\|\bm{v}_0\|+\|\varphi_0\|_{H^1}+\|\sigma_0\|\leq R_0$, it can be implied from the estimate \eqref{ener}, \eqref{sepa} and \eqref{h3} that the corresponding unique global weak solution $(\bm{v},\varphi, \sigma)$ fulfills
			\begin{equation}
			\|\bm{v}(t)\|_{\bm{H}^1}+\|\varphi(t)\|_{H^2}+\|\sigma(t)\|_{H^1}\leq R_1\quad \text{and}\ \ \|\varphi(t)\|_{L^\infty(\Omega)}\leq 1-\delta_1,\quad \forall\, t\geq 1,\label{sep}
			\end{equation}
			where the constants $R_1>0$ and $\delta_1\in (0,1)$ depend on coefficients of the system, $R_0$, $m_1$, $m_2$ and  $\Omega$. Thus, for every bounded set $\mathcal{B} \subset \mathcal{X}_{m_1,m_2}$, it can be indicated by taking $t_1(\mathcal{B}) = t_0(\mathcal{B})+1$ that for the ball $\mathcal{B}_1$ defined above, it holds  $\mathcal{S}(t)\mathcal{B}\subset \mathcal{B}_1$ when $t\geq t_1$.
	}}
\end{proof}
 \textbf{Step 2}. {\color{black}{
		From the definition of the absorbing set $\mathcal{B}_1$, it can be implied that there exists some $t_1(\mathcal{B}_1)>0$ satisfying $\mathcal{S}(t)\mathcal{B}_1\subset \mathcal{B}_1$, for any $t\geq t_1(\mathcal{B}_1)$. We define
		\be
		\begin{aligned}
		\mathcal{B}^*=\bigcup_{t\geq t_1(\mathcal{B}_1)}\mathcal{S}(t) \mathcal{B}_1,\quad\widetilde{\mathcal{B}}_1=\overline{\mathcal{B}^*}^{\bm{L}^2\times H^1\times L^2}.
		\end{aligned}
	\ee
		It is easy to prove that $\mathcal{B}^*\subset \mathcal{B}_1$. To this end, for any $(\bm{v}^*,\varphi^*, \sigma^*)\in \mathcal{B}^*$, there exist certain $(\bm{v}_0,\varphi_0, \sigma_0)\in \mathcal{B}_1$ and $t^*\geq t_1(\mathcal{B}_1)$ such that $(\bm{v}^*,\varphi^*, \sigma^*)=\mathcal{S}(t^*)(\bm{v}_0,\varphi_0, \sigma_0)\in \mathcal{B}_1$. Furthermore, for any $(\bm{v}^*,\varphi^*, \sigma^*)\in \mathcal{B}^*$, $\mathcal{S}(t)(\bm{v}^*,\varphi^*, \sigma^*)=\mathcal{S}(t)\mathcal{S}(t^*)(\bm{v}_0,\varphi_0, \sigma_0)=\mathcal{S}(t+t^*)(\bm{v}_0,\varphi_0, \sigma_0)\in \mathcal{B}^*$, which implies 
		$ \mathcal{S}(t) \mathcal{B}^*\subset \mathcal{B}^* $ for all $t\geq 0$. 
		Therefore, $\mathcal{B}^*$ is positively invariant in $\mathcal{X}_{m_1,m_2}$ satisfying $\mathcal{S}(t)\mathcal{B}^*\subset \mathcal{B}^*$ for all $t\geq 0$ and is relatively compact in $\mathcal{X}_{m_1,m_2}$.
		 As $\widetilde{\mathcal{B}}_1$ is a closed subset of $\mathcal{B}_1$, $\widetilde{\mathcal{B}}_1$ is also compact in $\mathcal{X}_{m_1,m_2}$. Furthermore, $\widetilde{\mathcal{B}}_1$ is also positively invariant under $\mathcal{S}(t)$. We have proved that $\mathcal{S}(t)$ is a strongly continuous
		 semigroup on $\mathcal{X}_{m_1,m_2}$ (see \eqref{conti}), thus for all $t\geq 0$, $\mathcal{S}(t)\widetilde{\mathcal{B}}_1\subset \overline{\mathcal{S}(t)\mathcal{B}^*} \subset \overline{\mathcal{B}^*}=\widetilde{\mathcal{B}}_1$ (with the corresponding closure taken in $\bm{L}^2(\Omega)\times H^1(\Omega)\times L^2(\Omega)$). As $\mathcal{S}(t)\mathcal{B}_1\subset \widetilde{\mathcal{B}}_1$ for all $t\geq t_1(\mathcal{B}_1)$, then  $\widetilde{\mathcal{B}}_1$ is a compact absorbing set for $\mathcal{S}(t)$.
		
		We observe that
		for any $(\bm{v},\varphi, \sigma)$ in $\widetilde{\mathcal{B}}_1$, it holds
		\begin{equation}
		\|\bm{v}\|_{\bm{H}^1}+\|\varphi\|_{H^2}+\|\sigma\|_{H^1}\leq R_1,\quad \text{and}\ \ \|\varphi\|_{L^\infty(\Omega)}\leq 1-\delta_1,\label{sep1}
		\end{equation}
		with the constants $R_1>0$ and $\delta_1\in (0,1)$. Combining with \eqref{h3}, it is easy to see that $\widetilde{\mathcal{B}}_1$ is bounded in $\bm{H}^1(\Omega)\times H^3(\Omega)\times H^1(\Omega)$.
		 In particular, as $\varphi$ satisfies the uniform separation property, then the singular potential $\varPsi(\varphi)$ can be considered as a globally Lipschitz smooth potential function. This fact admits us to prove the following smoothing properties:
		\bl\label{continu}
		The mapping $(t, (\bm{v}_0,\varphi_0,\sigma_0))\mapsto \mathcal{S}(t)(\bm{v}_0,\varphi_0,\sigma_0): [0, T] \times \widetilde{\mathcal{B}}_1 \to \widetilde{\mathcal{B}}_1$ is $\frac12$-H\"{o}lder continuous
		in time and Lipschitz continuous with respect to the initial data, when $\widetilde{\mathcal{B}}_1$ is endowed with the $\bm{L}^2(\Omega)\times H^1(\Omega)\times L^2(\Omega)$-topology. Furthermore, Assume that $(\bm{v}_i,\varphi_i,\sigma_i)$, $i=1,2$ are two global weak solutions to problem \eqref{f1.a}--\eqref{ini0} corresponding to the initial data $(\bm{v}_{0i},\varphi_{0i},\sigma_{0i}) \in \widetilde{\mathcal{B}}_1$, respectively. The difference of solutions satisfies the following smoothing estimate:
		\be
		\begin{aligned}
		&\| \bm{v}_1(t)-\bm{v}_2(t)\|_{\bm{H}^1}+\| \varphi_1(t)-\varphi_2(t)\|_{H^2}+\| \sigma_1(t)-\sigma_2(t)\|_{H^1}\\
		&\quad\leq C_T\left(\frac{1+t}{t}\right)(\| \bm{v}_{01}-\bm{v}_{02}\|+\|\varphi_{01}-\varphi_{02}\|_{H^1}+\|\sigma_{01}-\sigma_{02}\|),
		\label{smooth}
		\end{aligned}
	\ee
for any $t\in (0,T]$. Here, the positive constant $C_T$	depends on  the constants $R_1>0$, coefficients of the system, $R_0$, $\Omega$ and $T$.
		\el
\begin{proof}
 In the following proof, the symbols $C_i'$, $i\in \mathbb{N}$, denote generic positive constants that may depend on  the constants $R_1>0$, coefficients of the system, $R_0$, $\Omega$ and $T$. Their values may change from line to line.
			 
Denote the differences of functions by
\begin{align*}
&(\bm{v},p, \varphi,\mu,\sigma)
			=(\bm{v}_{1}-\bm{v}_{2},p_1-p_2, \varphi_{1}-\varphi_{2},\mu_1-\mu_2,\sigma_{1}-\sigma_{2}),\\
			&(\bm{v}_0, \varphi_0, \sigma_0)
			=(\bm{v}_{01}-\bm{v}_{02}, \varphi_{01}-\varphi_{02}, \sigma_{01}-\sigma_{02}).
			\end{align*}
			Then it holds
			\begin{subequations}
				\begin{alignat}{3}
				&\partial_t  \bm{ v}+\bm{ v}_1 \cdot \nabla  \bm {v} +\bm{ v} \cdot \nabla  \bm {v}_2-\textrm{div} (  2\eta(\varphi_1) D\bm{v} )-\textrm{div} \big(  2(\eta(\varphi_1)-\eta(\varphi_2)) D\bm{v}_2 \big)+\nabla p\notag\\
				&\quad=(\mu_1+\chi \sigma_1)\nabla \varphi+(\mu+\chi \sigma)\nabla \varphi_2\label{2atest33.c},\\
				&\partial_t  \varphi+\bm{v}_{1} \cdot \nabla \varphi+\bm{v}\cdot \nabla \varphi_{2} =\Delta \mu -a(x)(m(\varphi_1)-m(\varphi_2)),\label{2atest11.a} \\
				&\mu=A\varPsi'(\varphi_{1})-A\varPsi'(\varphi_{2})-B\Delta \varphi-\chi \sigma, \label{2atest44.d}\\
				&\partial_t  \sigma+\bm{v}_{1} \cdot \nabla \sigma+\bm{v} \cdot \nabla \sigma_{2}-\Delta \sigma  = -\chi\Delta \varphi, \label{2atest22.b}
				\end{alignat}
			\end{subequations}
			almost everywhere in $\Omega\times (0,T)$, and
			\begin{subequations}
				\begin{alignat}{3}
				&\bm{v}=\mathbf{0},\quad {\partial}_{\bm{n}}\varphi={\partial}_{\bm{n}}\sigma={\partial}_{\bm{n}}\mu=0, \qquad\qquad\qquad\qquad \qquad\qquad\ \textrm{on} \   \partial\Omega\times(0,T),\\
				&\bm{v}|_{t=0}=\bm{0},\quad \varphi|_{t=0}=0,\quad  \sigma|_{t=0}=0,\qquad\qquad\qquad \qquad\qquad\quad \textrm{in} \   \Omega.
				\end{alignat}
			\end{subequations}
			Multiplying \eqref{2atest33.c} by $\bm{v}$ and integrating over $\Omega$, we find
			\begin{align}
			\frac{1}{2} \frac{d}{dt}\|\boldsymbol{v}\|^{2} +2\big(\eta\left(\varphi_{1}\right) D \boldsymbol{v}, D \boldsymbol{v}\big)=J_1+J_2+J_3,
			\label{vt1}
			\end{align}
			where
			\begin{equation*}
			\begin{aligned}
			J_{1} &=-\left(\boldsymbol{v}_1 \cdot \nabla \boldsymbol{v}, \boldsymbol{v}\right) -\left(\boldsymbol{v} \cdot \nabla \boldsymbol{v}_{2}, \boldsymbol{v}\right)=-\left(\boldsymbol{v} \cdot \nabla \boldsymbol{v}_{2}, \boldsymbol{v}\right), \\ J_{2}&=-2\left(\left(\eta\left(\varphi_{1}\right)-\eta\left(\varphi_{2}\right)\right) D \boldsymbol{v}_{2}, \nabla \boldsymbol{v}\right), \\
			J_3&=((\mu_1+\chi \sigma_1)\nabla \varphi,\bm{v})+((\mu+\chi \sigma)\nabla \varphi_2,\bm{v}).
			\end{aligned}
			\end{equation*}
			From \cite[Section 5]{H1}, we have
			\be
			\begin{aligned}
			&\|\boldsymbol{v}_1(t)-\boldsymbol{v}_2(t)\|^{2}+\|\varphi_1(t)-\varphi_2(t)\|_{H^1}^{2} +\|\sigma_1(t)-\sigma_2(t)\|^{2}\\ &\quad+\int_{0}^t\Big(\|\boldsymbol{v}_1(s)-\boldsymbol{v}_2(s)\|^{2}_{\bm{H}^1} +\|\varphi_1(s)-\varphi_2(s)\|_{H^3}^{2} +\|\sigma_1(s)-\sigma_2(s)\|^{2}_{H^1}\Big)\, d s\\
			&\qquad\le C_1'\Big( \|\boldsymbol{v}_{0,1}-\boldsymbol{v}_{0,2}\|^{2} +\|\varphi_{0,1}-\varphi_{0,2}\|_{H^1}^{2}+\|\sigma_{0,1}-\sigma_{0,2}\|^{2}\Big),\quad \forall\,t\in[0,T].\label{bd1}
			\end{aligned}
		\ee
			
			Then we prove the smoothing estimate for the difference of solutions.
			Multiplying \eqref{2atest33.c} by $\bm{S}\bm{v}$ and integrating over $\Omega$, we obtain
			\be
			\begin{aligned}
				\frac{1}{2} \frac{d}{d t}\|\nabla \bm{v}\|^2&=\int_{\Omega}\textrm{div} (2  \eta(\varphi) D\bm{v} )\cdot\bm{S} \bm{v}\,dx -\int_{\Omega}\left(\bm{v}_1 \cdot \nabla\right) \bm{v} \cdot \bm{S} \bm{v} \,d x-\int_{\Omega}(\bm{v} \cdot \nabla) \bm{v}_2 \cdot \bm{S} \bm{v} \,d x\\
				&\quad+\int_{\Omega}2 \operatorname{div}\left(\left(\eta\left(\varphi_1\right)-\eta\left(\varphi_2\right)\right) D \bm{v}_2\right) \cdot \bm{S} \bm{v} \,d x+\int_{\Omega} (\mu_1+\chi\sigma_1) \nabla \varphi \cdot \bm{S} \bm{v} \,d x \\
				&\qquad+\int_{\Omega} (\mu+\chi\sigma) \nabla \varphi_2 \cdot \bm{S} \bm{v} \,d x .
			\end{aligned}
			\ee
			By Lemma \ref{stokes} for the Stokes operator $\bm{S}$, there exists $p \in L^2\left(0, T ; H^1(\Omega)\right)$ satisfying $-\Delta \boldsymbol{v}+\nabla p=\bm{S}\boldsymbol{v}$ and
			\be
			\|p\| \leq C\|\nabla \boldsymbol{v}\|^{\frac{1}{2}}\|\bm{S}\boldsymbol{v}\|^{\frac{1}{2}}, \quad\|p\|_{H^1} \leq C\|\bm{S}\boldsymbol{v}\|.\label{stoke}
			\ee
			Thus, we have
			$$
			\begin{aligned}
			&\int_{\Omega}\textrm{div} (2  \eta(\varphi) D\bm{v} )\cdot\bm{S} \bm{v}\,dx\\
			&\quad=\int_{\Omega} \eta(\varphi)\Delta \boldsymbol{v} \cdot \bm{S} \boldsymbol{v} \,d x+\int_{\Omega} 2\eta^{\prime}\left(\varphi_1\right)\left(D \bm{v} \nabla \varphi_1\right) \cdot \bm{S} \bm{v} \,dx \\
			&\quad=-\int_{\Omega} \eta(\varphi)|\bm{S} \boldsymbol{v}|^2 d\, x+\int_{\Omega} \eta(\varphi)\nabla p \cdot \bm{S} \boldsymbol{v} \,d x+\int_{\Omega} 2\eta^{\prime}\left(\varphi_1\right)\left(D \bm{v} \nabla \varphi_1\right) \cdot \bm{S} \bm{v} \,dx \\
			&\quad=-\int_{\Omega} \eta(\varphi)|\bm{S} \boldsymbol{v}|^2 \,d x-\int_{\Omega} \eta^{\prime}(\varphi) p \nabla \varphi \cdot \bm{S} \boldsymbol{v} \,d x+\int_{\Omega} 2\eta^{\prime}\left(\varphi_1\right)\left(D \bm{v} \nabla \varphi_1\right) \cdot \bm{S} \bm{v} \,dx.
			\end{aligned}
			$$
			We deduce from Young's inequalities and the Sobolev embedding theorem
			\be
			\begin{aligned}
				\left|-\int_{\Omega}\left(\bm{v}_1 \cdot \nabla\right) \bm{v} \cdot \bm{S} \bm{v} \,dx\right| & \leq\left\|\bm{v}_1\right\|_{L^4}\|\nabla \bm{v}\|_{L^4}\|\bm{S} \bm{v}\| \\
				& \leq C \|\nabla \bm{v}\|^{\frac{1}{2}}\|\bm{S} \bm{v}\|^{\frac{3}{2}} \\
				& \leq \frac{\eta_*}{8}\|\bm{S} \bm{v}\|^2+C_{2}^{\prime}\|\nabla \bm{v}\|^2.
			\end{aligned}
		\ee
	It can be implied from Agmon's inequality
		\begin{equation}
		\begin{aligned}
		\left|-\int_{\Omega}(\bm{v} \cdot \nabla) \bm{v}_2 \cdot \bm{S} \bm{v} \,dx\right| & \leq\|\bm{v}\|_{L^{\infty}}\left\|\nabla \bm{v}_2\right\|\|\bm{S} \bm{v}\| \\
		& \leq C \|\nabla \bm{v}\|^{\frac{1}{2}}\|\bm{S} \bm{v}\|^{\frac{3}{2}} \\
		& \leq \frac{\eta_*}{8}\|\bm{S} \bm{v}\|^2+C_{3}^{\prime}\|\nabla \bm{v}\|^2.
		\end{aligned}
		\end{equation}
		Noticing the separation property \eqref{sep1} for $\varphi$ and the assumption (H1), we deduce (see \cite{GT})
		\begin{equation}
		\begin{aligned}
		&\left|\int_{\Omega} 2\operatorname{div}\left(\left(\eta\left(\varphi_1\right)-\eta\left(\varphi_2\right)\right) D \bm{v}_2\right) \cdot \bm{S} \bm{v} \,dx\right| \\
		&\quad\leq \left|\int_{\Omega} (\eta\left(\varphi_1\right)-\eta\left(\varphi_2\right)) \Delta \bm{v}_2 \cdot \bm{S} \bm{v} \,dx\right| \\
		&\qquad+\mid \int_{\Omega}2(\eta^{\prime}(\varphi_1) D \bm{v}_2 \nabla \varphi+(\eta^{\prime}(\varphi_1)-\eta^{\prime}(\varphi_2)) D \bm{v}_2 \nabla \varphi_2) \cdot \bm{S} \bm{v} \,dx \mid\\
		&\quad\leq  C\|\varphi\|_{L^{\infty}}\|\Delta \boldsymbol{v}_2\|\|\bm{S} \boldsymbol{v}\|+C\|D \boldsymbol{v}_2\|_{L^4}\|\nabla \varphi\|_{L^4}\|\bm{S} \boldsymbol{v}\| \\
		&\qquad+C\|\varphi\|_{L^{\infty}}\left\|D \boldsymbol{v}_2\right\|\left\|\nabla \varphi_2\right\|_{L^{\infty}}\|\bm{S} \boldsymbol{v}\|\\
		&\quad\leq  \frac{\eta_*}{8}\|\bm{S} \bm{v}\|^2+C_{4}^{\prime}\left(1+\left\|\bm{S} \bm{v}_2\right\|^2\right)\|\Delta \varphi\|^2 .
		\end{aligned}
		\end{equation}
		We infer from Agmon's inequality 
		\begin{equation}
		\begin{aligned}
		&\left|\int_{\Omega} (\mu_1+\chi\sigma_1) \nabla \varphi \cdot \bm{S} \bm{v} \,dx+\int_{\Omega} (\mu+\chi\sigma) \nabla \varphi_2 \cdot \bm{S} \bm{v} \,dx\right| \\
		&\quad=\left|\int_{\Omega} \operatorname{div}\left(\nabla \varphi_1 \otimes \nabla \varphi+\nabla \varphi \otimes \nabla \varphi_2\right) \cdot \bm{S} \bm{v} \,dx\right| \\
		&\quad\leq C\left(\left\|\nabla \varphi_1\right\|_{L^{\infty}}+\left\|\nabla \varphi_2\right\|_{L^{\infty}}\right)\|\Delta \varphi\|\|\bm{S} \bm{v}\| \\
		&\qquad+C\left(\left\|\varphi_1\right\|_{W^{2,4}}+\left\|\varphi_2\right\|_{W^{2,4}}\right)\|\nabla \varphi\|_{L^4}\|\bm{S} \bm{v}\| \\
		&\quad\leq \frac{\eta_*}{8}\|\bm{S} \bm{v}\|^2+C_{5}^{\prime}\|\Delta \varphi\|^2,
		\end{aligned}
		\end{equation}
		and
		$$
		\begin{aligned}
		\left|\int_{\Omega} \eta^{\prime}\left(\varphi_1\right)\left(D \bm{v} \nabla \varphi_1\right) \cdot \bm{S} \bm{v} \,dx\right| & \leq C\|D \bm{v}\|\left\|\nabla \varphi_1\right\|_{L^{\infty}}\|\bm{S} \bm{v}\| \\
		& \leq \frac{\eta_*}{8}\|\bm{S} \bm{v}\|^2+C_6^{\prime}\|\nabla \bm{v}\|^2.
		\end{aligned}
		$$
		Recalling \eqref{stoke}, we obtain
		$$
		\begin{aligned}
		\left|-\int_{\Omega} \eta^{\prime}\left(\varphi_1\right) p \nabla \varphi_1 \cdot \bm{S} \bm{v} \,dx\right| & \leq C\|p\|\left\|\nabla \varphi_1\right\|_{L^{\infty}}\|\bm{S} \bm{v}\| \\
		& \leq C\|\nabla \bm{v}\|^{\frac{1}{2}}\|\bm{S} \bm{v}\|^{\frac{3}{2}} \\
		& \leq \frac{\eta_*}{8}\|\bm{S} \bm{v}\|^2+C_7^{\prime}\|\nabla \bm{v}\|^2.
		\end{aligned}
		$$
		Collecting all the above estimates together, we end up with the differential inequality
		$$
		\begin{aligned} 
		&\frac{d}{d t}\frac{1}{2}\|\nabla \bm{v}\| +\frac12\int_{\Omega} \eta(\varphi)|\bm{S} \boldsymbol{v}|^2 \,d x\\
		&\leq C_8^{\prime}\left(1+\left\|\bm{S} \bm{v}_2\right\|^2\right)\left(\frac{1}{2}\|\nabla \bm{v}\|^2+\frac{1}{2}\|\Delta \varphi\|^2\right) .
		\end{aligned}
		$$	
		Applying the classical elliptic estimate to the Neumann problem \eqref{2atest44.d},  $\partial_{\bm{n}} \mu=0$ on $\partial \Omega$ implies that $\partial_{\bm{n}} \Delta \varphi=0$ on $\partial \Omega$. Multiplying \eqref{2atest11.a} by $ \Delta^2 \varphi$ and integrating over $\Omega$, we derive that
		\be
		\begin{aligned}
		&\frac{1}{2} \frac{d}{d t}\|\Delta \varphi\|^2+\int_{\Omega}B\left|\Delta^2 \varphi\right|^2 \,d x\\
		&\quad=\int_{\Omega} A\Delta\left(\varPsi^{\prime}\left(\varphi_1\right)-\varPsi^{\prime}\left(\varphi_2\right)\right) \Delta^2 \varphi \,dx+\int_{\Omega} \nabla\left(\boldsymbol{v}_1 \cdot \nabla \varphi\right) \cdot \nabla \Delta \varphi \,dx \\
		&\qquad+\int_{\Omega} \nabla\left(\boldsymbol{v} \cdot \nabla \varphi_2\right) \cdot \nabla \Delta \varphi \,dx-\int_{\Omega} \chi \Delta\sigma \Delta^2 \varphi \,dx-\int_{\Omega}(a(x)(m(\varphi_1)-a(x)(m(\varphi_2))\Delta^2 \varphi\,d x .
		\end{aligned}
		\ee
		As $\varphi$ satisfies the separation property \eqref{sep} and $\varPsi\in C^{4}(-1,1)$, we deduce 
	   $$
		\begin{aligned}
		&\left|A\int_{\Omega} \Delta\left(\varPsi^{\prime}\left(\varphi_1\right)-\varPsi^{\prime}\left(\varphi_2\right)\right) \Delta^2 \varphi \,dx\right|\\
		 &\quad= \int_{\Omega}A(\varPsi^{\prime \prime}\left(\varphi_1\right) \Delta \varphi_1+\varPsi^{\prime \prime\prime}|\nabla\varphi_1|^2-\varPsi^{\prime \prime}\left(\varphi_1\right) \Delta \varphi_1-\varPsi^{\prime \prime\prime}|\nabla\varphi_1|^2)\Delta^2 \varphi \,dx \\
		&\quad\leq  \int_{\Omega}A\left|\left(\varPsi^{\prime \prime}\left(\varphi_1\right) \Delta \varphi+\left(\varPsi^{\prime \prime}\left(\varphi_1\right)-\varPsi^{\prime \prime}\left(\varphi_2\right)\right) \Delta \varphi_2\right) \Delta^2 \varphi\right| \,dx \\
		&\qquad+\int_{\Omega}A\left|\left(\varPsi^{\prime \prime \prime}\left(\varphi_1\right)\left(\left|\nabla \varphi_1\right|^2-\left|\nabla \varphi_2\right|^2\right)+\left(\varPsi^{\prime \prime \prime}\left(\varphi_1\right)-\varPsi^{\prime \prime \prime}\left(\varphi_2\right)\right)\left|\nabla \varphi_2\right|^2\right) \Delta^2 \varphi\right| \,dx \\
		&\quad\leq  C\|\Delta \varphi\|\left\|\Delta^2 \varphi\right\| +C\sup_{s\in[1-\delta_1,1+\delta_1]}\|\varPsi^{\prime \prime \prime}(s)\|_{L^{\infty}}\|\varphi\|_{L^{\infty}}\left\|\Delta \varphi_2\right\|\left\|\Delta^2 \varphi\right\| \\
		&\qquad+C\left(\left\|\nabla \varphi_1\right\|_{L^{\infty}}+\left\|\nabla \varphi_2\right\|_{L^{\infty}}\right)\|\nabla \varphi\|\left\|\Delta^2 \varphi\right\|\\ &\qquad\quad+C\sup_{s\in[1-\delta_1,1+\delta_1]}\|\varPsi^{\prime \prime\prime \prime}(s)\|_{L^{\infty}}\|\varphi\|_{L^{\infty}}\left\|\nabla \varphi_2\right\|_{L^{\infty}}^2\left\|\Delta^2 \varphi\right\|\\
		&\quad\le \frac{B}{8}\left\|\Delta^2 \varphi\right\|^2+C_9^{\prime}\|\Delta \varphi\|^2.
		\end{aligned}
		$$	
		It can be deduced from the Sobolev embedding theorem
		$$
		\begin{aligned}
		&\quad\left|\int_{\Omega} \nabla\left(\boldsymbol{v}_1 \cdot \nabla \varphi\right) \cdot \nabla \Delta \varphi \,dx\right| \\
		&\quad \leq\left\|\nabla \boldsymbol{v}_1\right\|\|\nabla \varphi\|_{L^4}\|\nabla \Delta \varphi\|_{L^4}+\left\|\boldsymbol{v}_1\right\|_{L^4}\|\varphi\|_{H^2}\|\nabla \Delta \varphi\|_{L^4} \\
		&\quad \leq C\|\Delta \varphi\|\left\|\Delta^2 \varphi\right\| \\
		&\quad \leq \frac{B}{8}\left\|\Delta^2 \varphi\right\|^2+C_{10}^{\prime}\|\Delta \varphi\|^2, 
		\end{aligned}
		$$
		 and 
		 $$
		 \begin{aligned}
		&\left|\int_{\Omega} \nabla\left(\boldsymbol{v} \cdot \nabla \varphi_2\right) \cdot \nabla \Delta \varphi \,dx\right| \\
		&\quad \leq\|\nabla \boldsymbol{v}\|\left\|\nabla \varphi_2\right\|_{L^{\infty}}\|\nabla \Delta \varphi\|+\|\boldsymbol{v}\|_{L^4}\left\|\varphi_2\right\|_{W^{2,4}}\|\nabla \Delta \varphi\| \\
		&\quad \leq \frac{B}{8}\left\|\Delta^2 \varphi\right\|^2+C_{11}'\|\nabla \boldsymbol{v}\|^2.
		\end{aligned}
		$$	
		Recalling (H3), we notice that
		\begin{align}
		-\int_{\Omega}(a(x)(m(\varphi_1)-a(x)(m(\varphi_2))\Delta^2 \varphi\,d x&\le \|a(x)(m(\varphi_1)-a(x)(m(\varphi_2)\|\|\Delta^2 \varphi\|\notag\\
		&\le \frac{B}{8}\left\|\Delta^2 \varphi\right\|^2+C\|\varphi\|^2.
		\end{align} 
		We deduce from Young's inequality that
		\be
		-\int_{\Omega} \chi \Delta\sigma \Delta^2 \varphi \,dx\le \frac{B}{8}\|\Delta^2 \varphi\|+\frac{4\chi^2}{B}\|\Delta\sigma\|^2.
		\ee
		From the above estimates together, we obtain
		\be
		\begin{aligned}
		&\frac{d}{d t}\left(\frac{1}{2}\|\nabla \boldsymbol{v}\|^2+\frac{1}{2}\|\Delta \varphi\|^2\right)+\frac{B}{2}\|\Delta^2 \varphi\|^2 \\
		&\leq C_{12}'\left(1+\left\|\mathbf{S} \boldsymbol{v}_2\right\|^2\right)\left(\frac{1}{2}\|\nabla \boldsymbol{v}\|^2+\frac{1}{2}\|\Delta \varphi\|^2\right)+\frac{4\chi^2}{B}\|\Delta\sigma\|^2.\label{me1}
		\end{aligned}
		\ee
			Multiplying \eqref{2atest22.b} by $-\Delta \sigma$ and integrating over $\Omega$, we obtain
			\be
			\begin{aligned}
			\frac{1}{2}\frac{d}{dt}\|\nabla \sigma\|^2+\|\Delta \sigma\|^2
			&=\chi \int_\Omega \Delta  \varphi \Delta \sigma\, dx\\
			&\leq \frac{1}{2}\|\Delta \sigma\|^2+ \frac{\chi^2}{2} \|\Delta   \varphi\|^2. \label{ddiffd1}
			\end{aligned}
		    \ee
			Multiplying \eqref{ddiffd1} by $\frac{B}{16\chi^2}$ (as the decoupling case $\chi=0$ is indeed easier, we only need to consider $\chi\neq 0$), adding \eqref{me1} and  \eqref{ddiffd1} together, we obtain
			\begin{align}
			&\frac12\frac{d}{dt} \Big(\frac{B}{16\chi^2}\|\nabla\bm{v}\|^2+\frac{B}{16\chi^2}\|\Delta \varphi\|^2+ \|\nabla \sigma\|^2 \Big)+\frac{B^2}{32\chi^2}\|\Delta^2 \varphi\|^2
			+ \frac14\| \Delta \sigma\|^2
			\notag\\
			&\quad \leq  C_{13}'\left(1+\left\|\bm{S} \bm{v}_2\right\|^2\right)\left(\frac{B}{16\chi^2}\|\nabla\bm{v}\|^2+\frac{B}{16\chi^2}\|\Delta \varphi\|^2+ \|\nabla \sigma\|^2\right) .
			\label{ddiffe}
			\end{align}
			As a consequence, we infer that
			\be
\begin{aligned}
&\frac12\frac{d}{dt} t\Big(\frac{B}{16\chi^2}\|\nabla\bm{v}\|^2+\frac{B}{16\chi^2}\|\Delta \varphi\|^2+ \|\nabla \sigma\|^2 \Big) +t(\frac{B^2}{32\chi^2}\|\Delta^2 \varphi\|^2
+ \frac14\| \Delta \sigma\|^2)
\\
&\quad \leq  C_{14}'t\left(1+\left\|\bm{S} \bm{v}_2\right\|^2\right)\left(\frac{B}{16\chi^2}\|\nabla\bm{v}\|^2+\frac{B}{16\chi^2}\|\Delta \varphi\|^2+ \|\nabla \sigma\|^2\right)\\
&\qquad+\frac12\Big(\frac{B}{16\chi^2}\|\nabla\bm{v}\|^2+\frac{B}{16\chi^2}\|\Delta \varphi\|^2+ \|\nabla \sigma\|^2 \Big) .
\label{ddiffe1}
\end{aligned}
\ee
			It can be implied from the Gronwall lemma that
			\be
			\begin{aligned}
			&\frac{t}{2}\Big(\frac{B}{16\chi^2}\|\nabla\bm{v}\|^2+\frac{B}{16\chi^2}\|\Delta \varphi\|^2+ \|\nabla \sigma\|^2 \Big) \\
			&\leq C_{15}' \int_0^{t}(s+1)\Big(\frac{B}{16\chi^2}\|\nabla\bm{v}\|^2+\frac{B}{16\chi^2}\|\Delta \varphi\|^2+ \|\nabla \sigma\|^2 \Big)\, d s,\label{eq}
			\end{aligned}
			\ee
			From \eqref{bd1} and \eqref{eq}, we prove the smoothing estimate:
			\be
			\begin{aligned}
				&\| \bm{v}_1(t)-\bm{v}_2(t)\|_{\bm{H}^1}+\| \varphi_1(t)-\varphi_2(t)\|_{H^2}+\| \sigma_1(t)-\sigma_2(t)\|_{H^1}\\
				&\quad\leq C_T\left(\frac{1+t}{t}\right)(\| \bm{v}_{01}-\bm{v}_{02}\|+\|\varphi_{01}-\varphi_{02}\|_{H^1}+\|\sigma_{01}-\sigma_{02}\|),
			\end{aligned}
			\ee
			for any $t\in (0,T]$. Here, the positive constant $C_T$	depends on  the constants $R_1>0$, coefficients of the system, $R_0$, $\Omega$ and $T$.
			For any $t\in (0,T]$, there exists a constant $C^*$ depending on the coefficients of the system, norms of the initial data, $\Omega$, $T$ and $t$ such that
			\be
			\begin{aligned}
			&\left\|S\left(t\right)\left(\boldsymbol{v}_{01}, \varphi_{01},\sigma_{01}\right)-S\left(t\right)\left(\boldsymbol{v}_{02}, \varphi_{02},\sigma_{02}\right)\right\|_{\bm{H}^1(\Omega) \times H^2(\Omega)\times H^1(\Omega)} \\
			&\quad\leq C^*\left\|\left(\boldsymbol{v}_{01}, \varphi_{01},\sigma_{01}\right)-\left(\boldsymbol{v}_{02}, \varphi_{02},\sigma_{02}\right)\right\|_{\mathcal{X}_{m_1,m_2}} .
			\end{aligned}
		\ee
			The proof is complete.
		\end{proof}
	\textbf{Step 3}. {\color{black}{
			Employing the abstract result \cite[Lemma 5.3]{BG} (we cite to \cite[Section 2]{GGMP06} for the more general version), we can imply from Lemma \ref{compatt} and Lemma \ref{continu} that there exists
			a closed bounded set $\mathcal{M}_{m_1,m_2}\subset\widetilde{\mathcal{B}}_1$, which is of finite fractal dimension in $\mathcal{X}_{m_1,m_2}$ and  positively invariant under $\mathcal{S}(t)$.
			Furthermore, there exist constants $\omega_0>0$ and
			$J_0>0$ such that
			\begin{align}
			\mathrm{dist}_{\bm{L}^2\times H^1\times L^2}(\mathcal{S}(t)\widetilde{\mathcal{B}}_1,\mathcal{M}_{m_1,m_2})\leq J_0 e^{-\omega_0 t},\quad \forall\, t\geq 0.\notag
			\end{align}
			For every bounded set $\mathcal{B} \subset \mathcal{X}_{m_1,m_2}$, there exists $t_2=t_2(\mathcal{B})>0$ satisfying $\mathcal{S}(t)\mathcal{B}\subset \widetilde{\mathcal{B}}_1$ for all $t\geq t_2(\mathcal{B})$ (notice that $\widetilde{\mathcal{B}}_1$ itself is a compact absorbing set). Thus,
			$$
			\mathrm{dist}_{\bm{L}^2\times H^1\times L^2}(\mathcal{S}(t)\mathcal{B},\mathcal{M}_{m_1,m_2})\leq J_0  e^{-\omega_0 [t-t_2(\mathcal{B})]},\quad \forall\, t\geq t_2(\mathcal{B}).
			$$
			It can be deduced from the above estimates and the boundedness of $\mathcal{M}_{m_1,m_2}$ that
			\begin{align*}
			&\mathrm{dist}_{\bm{L}^2\times H^1\times L^2}(\mathcal{S}(t)\mathcal{B},\mathcal{M}_{m_1,m_2})\\
			&\quad\leq \sup_{(\bm{v},\varphi,\sigma)\in \mathcal{B}}\sup_{t\in[0,\,t_2(\mathcal{B})]}\|\mathcal{S}(t)(\bm{v},\varphi,\sigma)\|_{\bm{L}^2\times H^1\times L^2}
			+ \sup_{(\bm{v},\varphi,\sigma)\in \mathcal{M}_{m_1,m_2}}\|(\bm{v},\varphi,\sigma)\|_{\bm{L}^2\times H^1\times L^2}\\
			&\quad\leq C_{\mathcal{B}},\quad \forall\, t\in[0,t_2(\mathcal{B})).
			\end{align*}
			Here, the constant $C_\mathcal{B}>0$ depends on the size of $\mathcal{B}$, $R_1$, $\Omega$, $m_1$, $m_2$ and coefficients of the system.
			As a consequence, for every bounded set $\mathcal{B} \subset \mathcal{X}_{m_1,m_2}$,  it holds
			\begin{align}
			\mathrm{dist}_{\bm{L}^2\times H^1\times L^2}(\mathcal{S}(t)\mathcal{B},\mathcal{M}_{m_1,m_2})\leq \big(\max\{J_0, C_\mathcal{B}\}e^{\omega_0t_2(\mathcal{B})}\big) e^{-\omega_0 t},\quad \forall\, t\geq 0.\notag
			\end{align}
			This implies that $\mathcal{M}_{m_1,m_2}$ attracts exponentially fast all bounded subsets $\mathcal{B}\subset \mathcal{X}_{m_1,m_2}$, that is,
			its basin of exponential attraction is the whole phase space $\mathcal{X}_{m_1,m_2}$.
			
			The proof of Theorem \ref{eattr} is complete. \hfill $\square$
}}

\section*{Acknowledgments}
\noindent
The author would like to thank Professor H. Wu for helpful discussions. The author is partially supported by Zhejiang Provincial Natural Science
Foundation of China (Grant Number LQ24A010011).


%
\end{document}